\documentclass[a4paper,11pt]{article}

\usepackage{a4wide}
\usepackage{graphicx}
\usepackage{amssymb,amsmath,amsthm,mathrsfs,xfrac,mathtools}
\usepackage{enumerate, enumitem}
\usepackage{amsfonts}
\usepackage{hyperref}
\usepackage{color}
\usepackage{pgf,tikz}
\usetikzlibrary{arrows}

\newtheorem{proposition}{Proposition}[section]
\newtheorem{theorem}[proposition]{Theorem}
\newtheorem{lemma}[proposition]{Lemma}
\newtheorem{definition}[proposition]{Definition}
\newtheorem{corollary}[proposition]{Corollary}
\newtheorem{remark}[proposition]{Remark}
\newtheorem{lgrthm}[proposition]{Algorithm}

\numberwithin{equation}{section}
\allowdisplaybreaks[4]

\renewenvironment{proof}{\smallskip\noindent\emph{\textbf{Proof.}}%
  \hspace{1pt}}{\hspace{-5pt}{\nobreak\quad\nobreak\hfill\nobreak%
    $\square$\vspace{2pt}\par}\smallskip\goodbreak}

\newenvironment{proofof}[1]{\smallskip\noindent{\textbf{Proof~of~#1.}}%
  \hspace{1pt}}{\hspace{-5pt}{\nobreak\quad\nobreak\hfill\nobreak%
    $\square$\vspace{2pt}\par}\smallskip\goodbreak}

\newcommand{\C}[1]{\mathbf{C^{#1}}}
\newcommand{\Cc}[1]{\mathbf{C_c^{#1}}}
\renewcommand{\L}[1]{\mathbf{L^#1}}

\newcommand{\BV}{\mathbf{BV}}

\newcommand{\modulo}[1]{{\left|#1\right|}}
\newcommand{\norma}[1]{{\left\|#1\right\|}}
\newcommand{\caratt}[1]{{\chi_{\strut#1}}}
\newcommand{\reali}{{\mathbb{R}}}
\newcommand{\R}{\mathbb R}

\newcommand{\N}{{\mathbb N}}
\newcommand{\interi}{{\mathbb{Z}}}

\renewcommand{\epsilon}{\varepsilon}
\renewcommand{\phi}{\varphi}

\renewcommand{\theta}{\vartheta}

\newcommand{\tv}{\mathinner{\rm TV}}

\newcommand{\sgnp}{\operatorname{sgn}^{+}}
\newcommand{\sgnm}{\operatorname{sgn}^{-}}

\renewcommand{\d}[1]{\mathinner{\mathrm{d}{#1}}}

\newcommand{\dt}{{\Delta t}}
\newcommand{\dx}{{\Delta x}}

\newcommand{\uh}[1]{u^{n}_{#1}}

\DeclareMathOperator{\sgn}{sgn}

\newcommand{\del}{\partial}

\newcommand{\iq}{\iiiint\limits_{\Pi_T\times\Pi_T}}
\newcommand{\id}{\iint\limits_{\Pi_T}}

\newcommand{\be}{\begin{equation}}
\newcommand{\ee}{\end{equation}}

\definecolor{ffqqqq}{rgb}{1.,0.,0.}
\definecolor{uuuuuu}{rgb}{0.26666666666666666,0.26666666666666666,0.26666666666666666}

\makeatletter
\let\@fnsymbol\@arabic
\makeatother

\makeatletter
\newcommand\appendix@section[1]{%
\refstepcounter{section}%
\orig@section*{Appendix \@Alph\c@section: #1}%
\addcontentsline{toc}{section}{Appendix \@Alph\c@section: #1}%
}
\let\orig@section\section
\g@addto@macro\appendix{\let\section\appendix@section}
\makeatother

\title{Well-posedness of general 1D Initial Boundary Value Problems for scalar balance laws}

\author{Elena Rossi
    \footnote{Inria Sophia Antipolis - M\'editerran\'ee, Universit\'e C\^ote d'Azur, Inria, CNRS, LJAD, 2004 route des Lucioles - BP 93, 06902 Sophia Antipolis Cedex, France. Email: \texttt{elena.rossi@inria.fr}} }
\date{}

\begin{document}
\maketitle

\begin{abstract}

  \noindent We focus on the initial boundary value problem for a
  general scalar balance law in one space dimension. Under rather
  general assumptions on the flux and source functions, we prove the
  well-posedness of this problem and the stability of its solutions
  with respect to variations in the flux and in the source terms. For
  both results, the initial and boundary data are required to be
  bounded functions with bounded total variation.  The existence of
  solutions is obtained from the convergence of a Lax--Friedrichs type
  algorithm with operator splitting. The stability result follows from
  an application of Kru\v{z}kov's doubling of variables technique,
  together with a careful treatment of the boundary terms.

  \medskip

  \noindent\textit{2010~Mathematics Subject Classification: 35L04,
    35L65, 65M12, 65N08}

  \medskip

\end{abstract}

\section{Introduction}

Consider the following general Initial-Boundary Value Problem (IBVP)
for a one dimensional scalar balance law on the bounded interval
$]a,b[ \subset \reali$
\begin{equation}\label{eq:original}
  \left\{
    \begin{array}{l@{\qquad}r@{\,}c@{\,}l}
      \del_t u + \d{}_x f(t,x,u) =g (t,x,u),
      & (t,x)
      &\in
      &I \times ]a,b[,
      \\
      u(0,x) = u_o(x),
      & x
      &\in
      &]a,b[,
      \\
      u(t,a) = u_a (t),
      & t
      &\in
      &I,
      \\
      u(t,b) = u_b (t),
      & t
      &\in
      &I,
    \end{array}
  \right.
\end{equation}
where $I = ]0,T[ $ for a positive $T$ and we introduce the notation
\begin{equation}
  \label{eq:Dx}
  \d{}_x f\! \left(t,x,u (t,x)\right)
  =
  \partial_x f\! \left(t,x,u (t,x)\right)
  +
  \partial_u f \!\left(t,x,u (t,x)\right) \, \partial_x u (t,x).
\end{equation}

Aim of the present work is to prove the well-posedness of~\eqref{eq:original} and the stability of its solutions with
respect to variations in the flux and in the source functions.

IBVPs for balance laws in several space dimensions were originally
studied by Bardos, le Roux and
N\'ed\'elec~\cite{BardosLerouxNedelec}. However, the existence and
uniqueness result proved in~\cite{BardosLerouxNedelec} is limited to
rather smooth initial data, namely functions of class $\C2$, while the
boundary datum is assumed to be zero. An extension of this result to
more general, although smooth, boundary data is carried out
in~\cite{ColomboRossi2015}.  Other contributions in the field are due
to~\cite{Otto1996}, see also~\cite[Chapter~2]{Malek_book}, and more
recently~\cite{Vovelle2002} and~\cite{Martin}.  In particular, in this
latter article, the author proves the well-posedness of an IBVP for a
multi-dimensional balance law with $\L\infty$ data. However, a
restrictive hypothesis on the flux and the source functions is needed,
in order to get a maximum principle on the solution.

In all the above cited references, the vanishing viscosity technique
is used to get the existence of solutions. Here, existence is obtained
by proving the convergence of a Lax--Friedrichs type numerical scheme,
together with operator splitting to account for the source term.  The
idea of the proof comes from~\cite{DFG, Vovelle2002}. It is remarkable how
the $\L\infty$ and total variation estimates on the solution obtained
in the present work (see Theorem~\ref{thm:main}) are more accurate
with respect to those presented in~\cite{ColomboRossi2015}, allowing
moreover for less regular data.

As far as it concerns the uniqueness, the Lipschitz continuous
dependence on initial and boundary data of solutions to general multiD
IBVP, proved in~\cite[Theorem~4.3]{ColomboRossi2015}, applies to the
present setting. Indeed, this result is valid in a generality wider
than that assumed to prove existence of solutions
in~\cite{ColomboRossi2015}: its proof follows directly form the
definition of solution, which requires initial and boundary data of
class $\L\infty \cap \BV$.

The investigation of stability results for IBVPs for balance laws has
begun only recently. At the present time, only partial results, namely
considering particular classes of equations, are available. For
instance, in the multi-dimensional case, \cite{multipop} presents a
stability estimate for a class of multiD linear conservation laws in a
bounded domain, with homogeneous boundary conditions and initial data
of class $\L\infty \cap \BV$.  In one space dimension, only
conservation laws with a flux function not explicitly dependent on the
space variable $x$ have been considered, see~\cite{M2AS2018,IBVP1D}.
Our result is more general: a stability estimate for 1D IBVPs for
balance laws with general flux and source functions. An adaptation of
the doubling of variables technique by Kru\v zkov, together with a
careful treatment of the boundary terms, allows to obtain the desired
result.

\smallskip

The paper is organised as follows. Section~\ref{sec:MR} presents the
assumptions needed throughout the paper, the definitions of solution
to problem~\eqref{eq:original} and the main analytic results, namely
the well-posedness of the problem and the stability
estimate. Section~\ref{sec:existence} is devoted to the introduction
and analysis of the numerical scheme. The estimates necessary to prove
its convergence, as well as the convergence result, constitute the
contribution of this section. Finally, Section~\ref{sec:stabflux}
contains the proofs of the main results.

\section{Main Results}
\label{sec:MR}

Throughout, we denote by $I$ the time interval $]0,T[$, for a positive
$T$, and we set
\begin{displaymath}
  \Sigma= [0,T] \times ]a,b[ \times \reali.
\end{displaymath}

Following~\cite{Martin, Vovelle2002}, we set
\begin{displaymath}
  \sgnp (s)
  =
  \begin{cases}
    1 & \mbox{if } s>0,
    \\
    0 & \mbox{if } s \leq 0,
  \end{cases}
  \qquad
  \sgnm (s)
  =
  \begin{cases}
    0 & \mbox{if } s \geq 0,
    \\
    -1 & \mbox{if } s < 0,
  \end{cases}
  \qquad
  \begin{array}{r@{\;}c@{\;}l}
    s^+
    & =
    & \max\{s, 0\} ,
    \\
    s^-
    & =
    & \max\{-s, 0\} .
  \end{array}
\end{displaymath}
In the rest of the paper, we will denote $\mathcal{I} (r,s)
= [\min\left\{r,s\right\}, \max\left\{r,s \right\}]$, for any $r, s
\in \R$.

We require the following assumptions
\begin{enumerate}[label={\bf(f)}]
\item \label{f} $f \in \C2 (\Sigma; \reali)$; $\partial_u f,
  \, \partial^2_{xu} f \in \L\infty (\Sigma; \reali)$.
\end{enumerate}

\begin{enumerate}[label={\bf(g)}]
\item \label{g} $g \in \C2 (\Sigma; \reali)$; $\partial_u g \in
  \L\infty (\Sigma;\reali)$.
\end{enumerate}
\begin{enumerate}[label={\bf(D)}]
\item \label{ID} $u_o \in (\L\infty \cap \BV) (]a,b[; \reali)$ and
  $u_a, \, u_b \in (\L\infty \cap \BV) (I; \reali)$.
\end{enumerate}

\noindent We introduce the constants
\begin{align}
  \label{eq:lipconst}
  L_f (t) = \ & \norma{\partial_u f}_{\L\infty ([0,t] \times [a,b]
    \times \reali; \reali)}, & L_g (t) = \ & \norma{\partial_u
    g}_{\L\infty ([0,t] \times [a,b] \times \reali; \reali)}.
\end{align}

Concerning the definition of solution to problem~\eqref{eq:original},
we refer below to the following extension
of~\cite[Definition~1]{Vovelle2002} presented in~\cite{Martin} for the
multi dimensional case.
\begin{definition}
  \label{def:MV}
  An \emph{MV--solution} to the IBVP~\eqref{eq:original} on the
  interval $[0,T[$ is a map $u \in \L\infty([0,T[ \times ]a,b[; \R)$
  such that, for all $\phi \in \Cc1 (]-\infty, T[ \times \reali;
  \R^+)$ and $k \in \R$,
  \begin{align}
    \nonumber
    & \int_0^{T} \int_{a}^{b} \left\{ \left(u (t,x)
        -k\right)^\pm \partial_t \phi (t,x) + \sgn^\pm\left(u (t,x) -
        k\right) \left[ f\left(t, x, u (t,x)\right) - f\left(t, x,
          k\right) \right] \, \partial_x \phi (t,x) \right.
    \\
    \nonumber
    & \qquad\qquad \left.  + \sgn^\pm\left(u (t,x) -k\right)
      \left(g \left(t,x,u (t,x)\right) - \partial_x f \left(t, x,
          k\right) \right) \phi (t,x) \right\} \d{x} \d{t}
    \\
    \label{eq:MV}
    & + \int_{a}^{b} \left(u_o(x)-k\right)^\pm \phi(0,x) \d{x} \\
    \nonumber & + L_f (T) \left( \int_0^{T} \left(u_a(t) - k
      \right)^\pm \phi(t,a) \d{t} + \int_0^{T} \left( u_b(t) - k
      \right)^\pm \phi(t,b) \d{t}\right) \geq 0,
  \end{align}
  where $L_f (T)$ is as in~\eqref{eq:lipconst}.
\end{definition}
We introduce a second definition of solution to
problem~\eqref{eq:original}.  This is an adaptation of the definition
of solution presented in~\cite{BardosLerouxNedelec} to the one
dimensional case, where the domain is an open bounded interval.
\begin{definition}
  \label{def:solBLN} A BLN--solution to the IBVP~\eqref{eq:original}
  on the interval $[0,T[$ is a map $u\in (\L\infty \cap \BV) ([0,T[
  \times]a,b[; \R)$ such that, for all
  $\phi\in\Cc1(]-\infty,T[\times\reali;\R^+)$ and $k\in\R$,
  \begin{align}
    \nonumber
    & \int_0^{T} \int_{a}^{b} \left\{ \modulo{u (t,x)
      -k} \partial_t \phi (t,x) + \sgn \left(u (t,x) - k\right)
      \left[ f\left(t, x, u (t,x)\right) - f\left(t, x, k\right)
      \right] \, \partial_x \phi (t,x) \right.
    \\
    \nonumber
    & \qquad\qquad \left.  + \sgn \left(u (t,x) -k\right)
      \left(g \left(t,x,u (t,x)\right) - \partial_x f \left(t, x,
      k\right) \right) \phi (t,x) \right\} \d{x} \d{t}
    \\\label{eq:BLN}
    & + \int_a^b \modulo{u_o (x) - k} \, \phi (0,x)
      \, \d{x}
    \\
    \nonumber
    & + \int_0^{T} \sgn \left( u_a(t) - k \right) \left[ f
      \left(t, a, u (t, a^+)\right) - f\left(t, a, k\right) \right]\,
      \phi(t,a) \d{t}
    \\
    \nonumber
    & - \int_0^{T} \sgn \left( u_b(t) - k \right) \left[
      f\left(t, b, u (t,b^-)\right) - f \left(t, b, k\right) \right]
      \,\phi(t,b) \d{t} \geq 0.
  \end{align}
\end{definition}

We remark that, for functions in $(\L\infty \cap \BV) ([0,T[\times
]a,b[; \reali)$, Definition~\ref{def:MV} and
Definition~\ref{def:solBLN} are equivalent, see~\cite{definizioni} for
further details.

\begin{remark}\label{rem:blnbc}{\rm{ In Definition~\ref{def:solBLN},
      to ensure that the traces of $u$ at the boundary $u(t,a^+)$ and
      $u(t,b^+)$ are well defined, we need the solution to be of
      bounded total variation. Moreover, we recall the well-known
      BLN--boundary conditions (\cite[Lemma~5.6]{definizioni}),
      linking the boundary data to the traces of the solution at the
      boundary: \begin{itemize}
    \item at $x=a$: for all $k \in \reali$, for a.e.~$t\in \,]0,T[$
      \begin{displaymath}
        \left(\sgn (u(t, a^+) - k) - \sgn (u_a(t) - k)\right)
        \left(f (t, a, u(t, a^+)) - f (t, a, k)\right) \leq 0,
      \end{displaymath}
      or, equivalently, for all
      $k \in \mathcal{I} (u (t,a^+), u_a (t))$ and a.e.~$t\in \,]0,T[$
      \begin{displaymath}
        \sgn (u(t, a^+) - u_a(t))
        \left(f (t, a, u(t, a^+)) - f (t, a, k)\right) \leq 0;
      \end{displaymath}
    \item at $x=b$: for all $k \in \reali$, for a.e.~$t\in\,]0,T[$
      \begin{displaymath}
        \left(\sgn (u(t, b^-) - k) - \sgn (u_b(t) - k)\right)
        \left(f (t, b, u(t, b^-)) - f (t, b, k)\right) \geq 0,
      \end{displaymath}
      or, equivalently, for all $k \in \mathcal{I} (u (t,b^-), u_b (t))$
      and a.e.~$t\in \,]0,T[$
      \begin{displaymath}
        \sgn (u(t, b^-) - u_b(t))
        \left(f (t, b, u(t, b^-)) - f (t, b, k)\right) \geq 0.
      \end{displaymath}
    \end{itemize}
  }}
\end{remark}

The following Theorem contains the existence and uniqueness result, as
well as some \emph{a priori} estimates on the solution to the
IBVP~\eqref{eq:original}.
\begin{theorem}
  \label{thm:main}
  Let~\ref{f}, \ref{g} and~\ref{ID} hold. Then, for all
  $T>0$, the IBVP~\eqref{eq:original} has a unique MV--solution
  $u \in (\L\infty \cap \BV) ([0,T[ \times
  ]a,b[;\reali)$. Moreover, the following estimates hold: for any
  $t \in [0,T[$ and $\tau \in \, ]0,t[$,
  \begin{align}
    \label{eq:mainLinf}
    & \norma{u (t)}_{\L\infty (]a,b[)} \leq \ \left(
      \max\left\{\norma{u_o}_{\L\infty (]a,b[)}, \,
      \norma{u_a}_{\L\infty ([0,t])} , \,\norma{u_b}_{\L\infty
      ([0,t])}\right\} + \mathcal{C}_1 (t) \, t \right)
      e^{\mathcal{C}_2 (t) \, t},
    \\
    \label{eq:mainTV}
    & \tv (u (t)) \leq \ e^{t \, \mathcal{C}_2 (t)} \left(\tv (u_o) +
      \tv (u_a; [0,t]) + \tv (u_b;[0,t]) + \mathcal{K}_2 (t) \,
      t\right),
    \\
    \label{eq:mainLipTime}
    & \norma{u (t) - u (t-\tau)}_{\L1 (]a,b[)} \leq \ \tau \,
      \left(\mathcal{C}_t (t) +\norma{g}_{\L\infty([0,t] \times [a,b]
      \times [-\mathcal{U}(t),\, \mathcal{U}(t)]} \right),
  \end{align}
  where $\mathcal{C}_1 (t), \, \mathcal{C}_2
  (t)$ are as in~\eqref{eq:c1}--\eqref{eq:c2}, $\mathcal{K}_2
  (t)$ is as in~\eqref{eq:K2} and $\mathcal{C}_t
  (t)$ is as in~\eqref{eq:Ct}, with
  $\alpha=\norma{\partial_u f}_{\L\infty
    ([0,t]\times[a,b]\times[-\mathcal{U} (t), \, \mathcal{U}
    (t)])}$, $\mathcal{U} (t)$ being as in~\eqref{eq:Ut}.
\end{theorem}

\noindent The proof of Theorem~\ref{thm:main} is postponed to
Section~\ref{sec:stabflux}.

\begin{remark} {\rm{
    The $\L\infty$ bound~\eqref{eq:mainLinf} provided in Theorem~\ref{thm:main} is optimal.
    \begin{itemize}
    \item In the case $f=f (u)$ and $g=0$, we get
      $\mathcal{C}_1(t)= 0$ and $\mathcal{C}_2(t)=0$,
      so~\eqref{eq:mainLinf} reduces to the well-known maximum
      principle, compare with~\cite[Chapter~2,
      Remark~7.33]{Malek_book}.
    \item The functions $\mathcal{C}_1 (t)$ and $\mathcal{C}_2 (t)$
      are clearly strictly related to the source term and to the space
      dependent flux. Consider, for instance,
      problem~\eqref{eq:original} with $u_o = 0$ and $u_b = 0$,
      $f(t, x, u) = - x$ and $g=0$. The solution is $u (t,x)=t$, and
      from~\eqref{eq:c1} we obtain $\mathcal{C}_1 (t)=1$, so
      that~\eqref{eq:mainLinf} now reads
      $\norma{u (t)}_{\L\infty (]a,b[)} \leq t $.
    \item Compare~\eqref{eq:mainLinf} with the $\L\infty$ estimate on
      the solution presented
      in~\cite[Formula~(2.5)]{ColomboRossi2015}. At a glance, one
      could notice that the present estimate is more
      accurate. Moreover, here the boundary data should be of class
      $\L\infty\cap \BV$, thus they are not required to be as regular
      as in~\cite[Theorem~2.7]{ColomboRossi2015}.

      Consider, for instance, the following example:
      \begin{align*}
        f(u)= \ & u(1-u), & g(u) = \ & 0, & u_b(t) = \ & 0, & u_o(x)=
                                                              \ & 0,
      \end{align*}
      and, as boundary datum in $x=a$, an increasing smooth function
      (at least of class $\C3$, as required
      by~\cite[Theorem~2.7]{ColomboRossi2015}) such that
      \begin{align*}
        u_a(0)= \ & 0, & u_a (t) = \ & 0.4 \mbox{ for } t > 0.1.
      \end{align*}
      Now compare the two $\L\infty$ estimates at a time $t > 0.1$.
      As already observed in the first point of this remark,
      $\mathcal{C}_1(t)=\mathcal{C}_2(t)=0$, and the
      estimate~\eqref{eq:mainLinf} reads
      \begin{displaymath}
        \norma{u(t)}_{\L\infty} \leq \norma{u_a}_{\L\infty([0,t])} = 0.4.
      \end{displaymath}
      On the other hand, with the notation of~\cite{ColomboRossi2015},
      $c_1=1$, $c_2=0$,
      $c_3 = \norma{\partial_{uu}^2 f}_{\L\infty}=2$, so that the
      estimate~\cite[Formula~(2.5)]{ColomboRossi2015} reads
      \begin{displaymath}
        \norma{u(t)}_{\L\infty} \leq
        \left(t \, \norma{u_a}_{\L\infty}  + \norma{\partial_t u_a}_{\L\infty}
          +2  \,  \norma{u_a}_{\C2}\right) e^{t(1+2\norma{u_a}_{\C2})} + \norma{u_a}_{\L\infty},
      \end{displaymath}
      where all the norms on the right hand side of the inequality are
      evaluated on $[0,t]$.
    \end{itemize}
  }}
\end{remark}

The following Theorem presents a stability estimate with respect to
the flux and the source functions. A particular case of the
IBVP~\eqref{eq:original} is considered, for instance,
in~\cite{IBVP1D}, where a flux function independent of the space
variable is taken into account. There, a stability estimate with
respect to the flux function is provided.
\begin{theorem}
  \label{thm:stab}
  Let $f_1, \, f_2$ satisfy~\ref{f}, $g_1, \, g_2$ satisfy~\ref{g},
  $(u_o, \, u_a, \, u_b)$ satisfy~\ref{ID}. Call $u_1$ and $u_2$ the
  corresponding solutions to the IBVP~\eqref{eq:original}. Then, for
  all $t\in [0,T[$, the following estimate holds
  \begin{align*}
    & \norma{u_1 (t) - u_2 (t)}_{\L1 (]a,b[)}
    \\ \leq \
    & \exp\left(t \,
      \min\left\{\norma{\partial_u g_1}_{\L\infty ([0,t]\times[a,b] \times U (t))}, \,
      \norma{\partial_u g_2}_{\L\infty([0,t]\times[a,b] \times U (t))}\right\} \right)
    \\
    & \times \left( \int_0^t \int_a^b \norma{\partial_x (f_2-f_1)
      (s,x, \cdot)}_{\L\infty (U(s))} \d{x}\d{s} \right.
    \\
    & \qquad+ \int_0^t \int_a^b \norma{(g_1-g_2) (s,x,
      \cdot)}_{\L\infty (U (s))} \d{x}\d{s}
    \\
    & \qquad+ \int_0^t \norma{\partial_u (f_2-f_1) (s, \cdot,
      \cdot)}_{\L\infty (]a,b[\times U (s))} \, \min\left\{\tv
      \left(u_1 (s)\right), \,\tv \left(u_2 (s)\right)\right\} \d{s}
    \\
    & \left. \qquad+ 2 \int_0^t \norma{(f_2-f_1)(s,a,
      \cdot)}_{\L\infty(U (s))} \d{s} +2 \int_0^t
      \norma{(f_2-f_1)(s,b, \cdot)}_{\L\infty(U (s))} \d{s}\right),
  \end{align*}
  where $U(s)$ is as in~\eqref{eq:30}.
\end{theorem}
\noindent The proof is postponed to Section~\ref{sec:stabflux}.

\section{Existence of weak entropy solutions}
\label{sec:existence}

Consider a space step $\dx$ such that $b-a = N \Delta x$, $N \in\N$,
and a time step $\dt$ subject to a CFL condition which will be
specified later.  For $j=1,\ldots,N$, introduce the following notation
\begin{align*}
  y_{k} = \ & (k-1/2) \dx & y_{k+1/2} = \ & k \dx & \mbox{for } & k
                                                                  \in \interi,
\end{align*}
and let $x_{j+1/2}=a + j \dx = a + y_{j + 1/2}$ be the cells
interfaces, for $j=0, \ldots, N$, and $x_j= a + (j-1/2)\dx = a + y_j$
the cells centres, for $j=1, \ldots, N$. Moreover, set
$N_T = \lfloor T/\dt \rfloor$ and, for $n=0, \ldots, N_T$, let
$t^n=n\dt$ be the time mesh.  Set $\lambda = \dt / \dx$ and let
$\alpha \geq 1$ be the viscosity coefficient.

Approximate the initial datum $u_o$ and the boundary data as follows:
\begin{align*}
  &u_j^0 = \frac{1}{\dx} \int_{x_{j-1/2}}^{x_{j+1/2}} u_o(x) \d{x},
    \qquad j=1,\ldots,N,
  \\
  &u_a^n = \frac{1}{\dt} \int_{t^n}^{t^{n+1}} u_a(t) \d{t}, \quad
    u_b^n = \frac{1}{\dt} \int_{t^n}^{t^{n+1}} u_b(t) \d{t}, \qquad n=0,
    \ldots, N_T-1.
\end{align*}
Introduce moreover the notation $u_0^n = u_a^n$ and
$u_{N+1}^n = u_b^n$.

We define a piecewise constant approximate solution $u_{\Delta}$
to~\eqref{eq:original} as
\begin{equation}
  \label{eq:udelta}
  u_\Delta (t,x) = u_j^n \quad \mbox{ for } \quad
  \left\{
    \begin{array}{r@{\;}c@{\;}l}
      t & \in &[t^n, t^{n+1}[ \,,
      \\
      x & \in & [x_{j-1/2}, x_{j+1/2}[ \,,
    \end{array}
  \right.
  \quad \mbox{ where } \quad
  \begin{array}{r@{\;}c@{\;}l}
    n & = & 0, \ldots, N_T-1,
    \\
    j & = &  1,\ldots,N,
  \end{array}
\end{equation}
through a Lax--Friedrichs type scheme together with operator
splitting, in order to treat the source term.

In particular, the algorithm is defined as follows:
\begin{lgrthm}
  \label{alg:1}
  \begin{align}
    &\texttt{for } n=0,\ldots N_T-1 \nonumber
    \\
    \nonumber & \quad \texttt{for } j=1, \ldots, N
    \\
    & \quad\quad F^n_{j+1/2} (\uh{j}, \uh{j+1}) = \frac{1}{2} \left[ f
      (t^n, x_{j+1/2}, \uh{j}) + f (t^n, x_{j+1/2}, \uh{j+1}) \right]
      - \frac{\alpha}{2} \left(\uh{j+1} - \uh{j} \right)
      \label{eq:numFl}
    \\
    & \quad \quad u_j^{n+1/2} = u_j^n - \lambda \left(F_{j+1/2}^n
      (\uh{j}, \uh{j+1}) - F_{j-1/2}^n (\uh{j-1}, \uh{j})\right)
      \label{eq:scheme1}
    \\
    & \quad \quad u_j^{n+1} = u_j^{n+1/2} + g
      \left(t^n,x_j,u_j^{n+1/2}\right) \, \dt
      \label{eq:scheme2}
    \\
    &\quad \texttt{end}\nonumber
    \\
    & \texttt{end} \nonumber
  \end{align}
\end{lgrthm}
\noindent We require, moreover, the following CFL condition:
\begin{align}
  \label{eq:CFL}
  \alpha \geq \ & L_f (T), & \lambda \leq\frac{1}{3 \, \alpha}.
\end{align}

\noindent The proof of the convergence of approximate solutions
consists of several steps, whose aim is to show that the sequence
verifies the hypotheses of Helly's compactness theorem.

\subsection{\texorpdfstring{$\L\infty$}{L infinity} bound}
\label{sec:linf}

\begin{lemma}
  \label{lem:linf}
  Let~\ref{f}, \ref{g}, \ref{ID} and~\eqref{eq:CFL} hold.  Then, for
  all $t\in ]0,T[$, $u_\Delta$ in~\eqref{eq:udelta} defined through
  Algorithm~\ref{alg:1} satisfies
  \begin{equation}
    \label{eq:linf}
    \norma{u_\Delta (t, \cdot)}_{\L\infty (]a,b[)}
    \leq \mathcal{U} (t)
  \end{equation}
  where
  \begin{equation}
    \label{eq:Ut}
    \mathcal{U} (t) =  \left(
      \max\left\{\norma{u_o}_{\L\infty (]a,b[)}, \,
        \norma{u_a}_{\L\infty ([0,t])} , \,\norma{u_b}_{\L\infty ([0,t])}\right\} + t \, \mathcal{C}_1 (t)\right)
    e^{\mathcal{C}_2 (t) \, t},
  \end{equation}
  with
  \begin{align}
    \label{eq:c1}
    \mathcal{C}_1 (t) = \
    & \norma{\partial_x f (\cdot, \cdot,
      0)}_{\L\infty ([0,t] \times [a,b])} + \norma{g (\cdot, \cdot,
      0)}_{\L\infty ([0,t] \times [a,b])},
    \\ \label{eq:c2}
    \mathcal{C}_2 (t) = \
    & \norma{\partial^2_{xu} f }_{\L\infty
      ([0,t] \times [a,b] \times \reali)} + \norma{\partial_u
      g}_{\L\infty ([0,t] \times [a,b] \times \reali)}.
  \end{align}
\end{lemma}

\begin{proof}
  Fix $j$ between $1$ and $N$, $n$ between $0$ and $N_T-1$, and
  rewrite~\eqref{eq:scheme1} as follows:
  \begin{align}
    \nonumber
    & u_j^{n+1/2}
    \\ \nonumber = \
    & \uh{j} - \lambda \left[
      F_{j+1/2}^n (\uh{j}, \uh{j+1}) \pm F_{j+1/2}^n (\uh{j}, \uh{j})
      \pm F_{j-1/2}^n (\uh{j}, \uh{j})- F_{j-1/2}^n (\uh{j-1}, \uh{j})
      \right]
    \\ \label{eq:1} = \
    & (1 - \beta_j^n - \gamma_j^n) \,
      \uh{j} + \beta_j^n \, \uh{j-1} + \gamma_j^n \, \uh{j+1} - \lambda
      \left(F_{j+1/2}^n (\uh{j}, \uh{j}) - F_{j-1/2}^n (\uh{j}, \uh{j})
      \right),
  \end{align}
  with
  \begin{align}
    \label{eq:betajn}
    \beta_j^n = \
    &
      \begin{cases}
        \lambda \dfrac{F_{j-1/2}^n (\uh{j}, \uh{j})- F_{j-1/2}^n
          (\uh{j-1}, \uh{j})}{\uh{j}-\uh{j-1}} & \mbox{if } \uh{j}
        \neq \uh{j-1},
        \\
        0 & \mbox{if } \uh{j} = \uh{j-1},
      \end{cases}
    \\
    \label{eq:gammajn}
    \gamma_j^n = \
    &
      \begin{cases}
        - \lambda \dfrac{F_{j+1/2}^n (\uh{j}, \uh{j+1})- F_{j+1/2}^n
          (\uh{j}, \uh{j})}{\uh{j+1}-\uh{j}} & \mbox{if } \uh{j} \neq
        \uh{j+1},
        \\
        0 & \mbox{if } \uh{j} = \uh{j+1}.
      \end{cases}
  \end{align}
  Using the explicit expression of the numerical flux~\eqref{eq:numFl}
  and the hypotheses on $f$~\ref{f}, we observe that, whenever
  $\uh{j} \neq \uh{j-1}$,
  \begin{align*}
    \beta_j^n = \
    &\frac{\lambda}{2 \, (\uh{j}-\uh{j-1})}\left[ f
      (t^n, x_{j-1/2}, \uh{j}) - f (t^n, x_{j-1/2}, \uh{j-1}) + \alpha
      \, (\uh{j}-\uh{j-1}) \right]
    \\
    = \
    & \frac{\lambda}{2} \left(\partial_u f (t^n, x_{j-1/2},
      r_{j-1/2}^n) + \alpha \right),
  \end{align*}
  with $r_{j-1/2}^n \in \mathcal{I}\left(\uh{j-1},
    \uh{j}\right)$. Similarly, whenever $\uh{j} \neq \uh{j+1}$,
  \begin{align*}
    \gamma_j^n = \
    & -\frac{\lambda}{2 \, (\uh{j+1}-\uh{j})}\left[ f
      (t^n, x_{j+1/2}, \uh{j+1}) - f (t^n, x_{j+1/2}, \uh{j}) - \alpha
      \, (\uh{j+1}-\uh{j}) \right]
    \\
    = \
    & \frac{\lambda}{2} \left(\alpha - \partial_u f (t^n,
      x_{j+1/2}, r_{j+1/2}^n) \right),
  \end{align*}
  with $r_{j+1/2}^n \in \mathcal{I}\left(\uh{j}, \uh{j+1}\right)$.
  By~\eqref{eq:CFL} we get
  \begin{align}
    \label{eq:5}
    \beta_j^n, \, \gamma_j^n \in
    & \left[0,\frac13\right],
    & (1-\beta_j^n - \gamma_j^n) \in
    & \left[\frac13, 1\right].
  \end{align}
  Compute now
  \begin{align*}
    & \modulo{F_{j+1/2}^n (\uh{j}, \uh{j}) - F_{j-1/2}^n (\uh{j},
      \uh{j}) }
    \\
    = \
    & \modulo{ f (t^n, x_{j+1/2}, \uh{j}) - f (t^n, x_{j-1/2},
      \uh{j})}
    \\
    \leq \
    & \modulo{\partial_x f (t^n, \tilde{x}_j, \uh{j})
      \pm \partial_x f (t^n, \tilde{x}_j,0)} |x_{j+1/2} - x_{j-1/2}|
    \\
    \leq \
    & \dx \, \modulo{\partial_x f (t^n, \tilde{x}_j, 0)} + \dx
      \, \modulo{\uh{j}} \norma{\partial^2_{xu} f }_{\L\infty (([0,t^n]
      \times [a,b] \times \reali;\reali)}
    \\
    \leq \
    & \dx \, \norma{\partial_x f (\cdot, \cdot, 0)}_{\L\infty
      ([0,t^n] \times [a,b])} + \dx \, \modulo{\uh{j}}
      \norma{\partial^2_{xu} f }_{\L\infty (([0,t^n] \times [a,b] \times
      \reali)} ,
  \end{align*}
  where $\tilde{x}_j \in ]x_{j-1/2}, x_{j+1/2}[$. Observe that, thanks
  to~\ref{f}, we have
  $\norma{\partial_x f (\cdot, \cdot, 0)}_{\L\infty ([0,t] \times
    [a,b]; \reali)}< + \infty$ for all $t \in I$. Inserting the above
  estimate into~\eqref{eq:1} and exploiting the bounds~\eqref{eq:5} on
  $\beta_j^n$ and $\gamma_j^n$, we get
  \begin{align*}
    u_j^{n+1/2} \leq \
    & (1 - \beta_j^n - \gamma_j^n) \, \modulo{
      \uh{j}} + \beta_j^n \,\modulo{ \uh{j-1}} + \gamma_j^n \,
      \modulo{ \uh{j+1}} + \lambda \modulo{F_{j+1/2}^n (\uh{j}, \uh{j})
      - F_{j-1/2}^n (\uh{j}, \uh{j})}
    \\
    \leq \
    & (1 - \beta_j^n - \gamma_j^n) \, \norma{u^n}_{\L\infty
      (]a,b[)} + \beta_j^n \, \max\left\{\norma{u^n}_{\L\infty
      (]a,b[)}, \, \modulo{u_a^n} \right\}
    \\
    & + \gamma_j^n \, \max\left\{\norma{u^n}_{\L\infty (]a,b[)}, \,
      \modulo{u_b^n}\right\}
    \\
    & + \lambda \, \dx \left( \norma{\partial_x f (\cdot, \cdot,
      0)}_{\L\infty ([0,t^n] \times [a,b])} + \modulo{\uh{j}}
      \norma{\partial^2_{xu} f }_{\L\infty ([0,t^n] \times [a,b]
      \times \reali)}\right)
    \\
    \leq \
    & \max\left\{\norma{u^n}_{\L\infty (]a,b[)}, \,
      \modulo{u_a^n} , \,\modulo{u_b^n}\right\} \left(1 + \dt \,
      \norma{\partial^2_{xu} f }_{\L\infty ([0,t^n] \times [a,b]
      \times \reali)}\right)
    \\
    & + \dt \, \norma{\partial_x f (\cdot, \cdot, 0)}_{\L\infty
      ([0,t^n] \times [a,b])}
    \\
    \leq \
    & \!\left(\!  \max\left\{\norma{u^n}_{\L\infty (]a,b[)}, \,
      \norma{u_a}_{\L\infty ([0,t^n])}, \,\norma{u_b}_{\L\infty
      ([0,t^n])}\right\}\!  + \dt \, \norma{\partial_x f (\cdot,
      \cdot, 0)}_{\L\infty ([0,t^n] \times [a,b])} \right)
    \\
    & \times \exp\left(\dt \, \norma{\partial^2_{xu} f }_{\L\infty
      ([0,t^n] \times [a,b] \times \reali)}\right) .
  \end{align*}
  By~\eqref{eq:scheme2}, since, thanks to~\ref{g},
  $\norma{g (\cdot, \cdot, 0)}_{\L\infty ([0,t] \times [a,b])}<
  +\infty$ for all $t \in I$, we have
  \begin{align*}
    u_j^{n+1} = \
    & u_j^{n+1/2} + \dt \left(g
      \left(t^n,x_j,u_j^{n+1/2}\right) \pm g
      \left(t^n,x_j,0\right)\right)
    \\
    \leq \
    & \modulo{u_j^{n+1/2}} + \dt \modulo{u_j^{n+1/2}}
      \norma{\partial_u g}_{\L\infty ([0,t^n] \times [a,b] \times
      \reali)} + \dt \norma{g (\cdot, \cdot, 0)}_{\L\infty ([0,t^n]
      \times [a,b])}
    \\
    \leq \
    & \modulo{u_j^{n+1/2}} \left( 1 + \dt \,\norma{\partial_u
      g}_{\L\infty ([0,t^n] \times [a,b] \times \reali)} \right) +
      \dt \norma{g (\cdot, \cdot, 0)}_{\L\infty ([0,t^n] \times [a,b])}
    \\
    \leq \
    & \left( \modulo{u_j^{n+1/2}} + \dt \norma{g (\cdot, \cdot,
      0)}_{\L\infty ([0,t^n] \times [a,b])} \right) \exp\left( \dt
      \,\norma{\partial_u g}_{\L\infty ([0,t^n] \times [a,b] \times
      \reali)}\right)
    \\
    \leq \
    & \Bigl( \max\left\{\norma{u^n}_{\L\infty (]a,b[)}, \,
      \norma{u_a}_{\L\infty ([0,t^n])} , \,\norma{u_b}_{\L\infty
      ([0,t^n])}\right\}
    \\
    & \quad+ \dt \, \norma{\partial_x f (\cdot, \cdot, 0)}_{\L\infty
      ([0,t^n] \times [a,b])} + \dt \norma{g (\cdot, \cdot,
      0)}_{\L\infty ([0,t^n] \times [a,b])} \Bigr)
    \\
    & \times\exp\left(\dt\left( \norma{\partial^2_{xu} f }_{\L\infty
      ([0,t^n] \times [a,b] \times \reali)} + \norma{\partial_u
      g}_{\L\infty ([0,t^n] \times [a,b] \times \reali)}
      \right)\right).
  \end{align*}
  An iterative argument yields the thesis.
\end{proof}

\subsection{BV estimates}
\label{sec:bv}

\begin{proposition}\label{prop:BV} {\bf ($\BV$ estimate in space)}
  Let~\ref{f}, \ref{g}, \ref{ID} and~\eqref{eq:CFL} hold.  Then, for
  $n$ between $1$ and $N_T$, the following estimate holds
  \begin{equation}
    \label{eq:spaceBV}
    \sum_{j=0}^{N}
    \modulo{u^{n}_{j+1}-u^{n}_{j}} \leq
    \mathcal{C}_x (t^n),
  \end{equation}
  where
  \begin{equation}
    \label{eq:11}
    \mathcal{C}_x (t^n) =
    e^{\mathcal{C}_2 (t^n) \, t^n}  \left(
      \sum_{j=0}^N \modulo{u_{j+1}^{0}  -u_j^{0}}
      +  \sum_{m=1}^n \modulo{u_{a}^{m}  -u_a^{m-1}}
      +  \sum_{m=1}^n \modulo{u_{b}^{m}  -u_b^{m-1}}
      + t^n \, \mathcal{K}_2 (t^n)\right),
  \end{equation}
  with $\mathcal{C}_2 (t^n)$ defined as in~\eqref{eq:c2} and
  $\mathcal{K}_2 (t^n)$ defined as in~\eqref{eq:K2}.
\end{proposition}

\begin{proof}
  For the sake of simplicity, introduce the space
  $\Sigma_n = [0,t^n] \times [a,b] \times [-\mathcal{U} (t^n),
  \mathcal{U} (t^n)]$, with the notation introduced
  in~\eqref{eq:Ut}. By the definition of the
  scheme~\eqref{eq:scheme2}, observe that, for all $j=1, \ldots, N-1$,
  \begin{align*}
    \modulo{u^{n+1}_{j+1} - u^{n+1}_j} = \
    & \modulo{u^{n+1/2}_{j+1} -
      u^{n+1/2}_j + \dt \, \left(g \left(t^n, x_{j+1},
      u^{n+1/2}_{j+1}\right) - g \left(t^n, x_j,
      u^{n+1/2}_{j}\right)\right)}
    \\
    \leq \
    & \modulo{u^{n+1/2}_{j+1} - u^{n+1/2}_j} \, \left(1 + \dt
      \,\norma{\partial_u g}_{\L\infty (\Sigma_n)}\right) + \dt \, \dx
      \, \norma{\partial_xg}_{\L\infty (\Sigma_n)},
  \end{align*}
  where $\norma{\partial_xg}_{\L\infty (\Sigma_n)}$ is bounded, thanks
  to~\ref{g}.  On the other hand, for $j=0$ and $j=N$ we have
  respectively
  \begin{align*}
    \modulo{u^{n+1}_1 - u^{n+1}_a} = \
    & \modulo{u^{n+1/2}_1
      -u^{n+1}_a + \dt \, g \left(t^n, x_1, u^{n+1/2}_1\right) \pm \dt
      \, g\left(t^n,x_1,0\right)}
    \\
    \leq \
    & \modulo{u^{n+1/2}_1 -u^{n+1}_a} + \dt \,
      \norma{\partial_u g}_{\L\infty (\Sigma_n)} \modulo{u^{n+1/2}_1} +
      \dt \norma{g (\cdot, \cdot, 0)}_{\L\infty ([0,t^n]\times [a,b])},
    \\
    \modulo{u^{n+1}_b - u^{n+1}_N} = \
    & \modulo{u^{n+1/2}_b
      -u^{n+1}_N - \dt \, g \left(t^n, x_N, u^{n+1/2}_N\right) \pm \dt
      \, g\left(t^n,x_N,0\right)}
    \\
    \leq \
    & \modulo{u^{n+1/2}_b -u^{n+1}_N} + \dt \,
      \norma{\partial_u g}_{\L\infty (\Sigma_n)} \modulo{u^{n+1/2}_N} +
      \dt \norma{g (\cdot, \cdot, 0)}_{\L\infty ([0,t^n]\times [a,b])}.
  \end{align*}
  Therefore,
  \begin{align}
    \nonumber
    & \sum_{j=0}^N\modulo{u^{n+1}_{j+1} - u^{n+1}_j}
    \\
    \nonumber \leq \
    & e^{\dt \, \norma{\partial_u g}_{\L\infty
      (\Sigma_n)}} \sum_{j=1}^{N-1} \modulo{u^{n+1/2}_{j+1} -
      u^{n+1/2}_j} + \dt \, (b-a) \, \norma{\partial_xg}_{\L\infty
      (\Sigma_n)}
    \\ \nonumber
    & + \modulo{u^{n+1/2}_1 -u^{n+1}_a} +
      \modulo{u^{n+1/2}_b -u^{n+1}_N} + 2 \dt \norma{g (\cdot, \cdot,
      0)}_{\L\infty ([0,t^n]\times [a,b])}
    \\ \nonumber
    & + 2 \, \dt
      \, \norma{\partial_u g}_{\L\infty (\Sigma_n)}
      \norma{u^{n+1/2}}_{\L\infty ([a,b])}
    \\ \nonumber
    \leq \
    & e^{\dt
      \, \norma{\partial_u g}_{\L\infty (\Sigma_n)}} \left(
      \sum_{j=1}^{N-1} \modulo{u^{n+1/2}_{j+1} - u^{n+1/2}_j} +
      \modulo{u^{n+1/2}_1 -u^{n+1}_a} + \modulo{u^{n+1/2}_b
      -u^{n+1}_N} \right)
    \\ \label{eq:4}
    & + \dt \, (b-a) \,
      \norma{\partial_xg}_{\L\infty (\Sigma_n)}
    \\ \nonumber
    & + 2 \dt
      \left( \norma{g (\cdot, \cdot, 0)}_{\L\infty ([0,t^n]\times
      [a,b])}+ \norma{\partial_u g}_{\L\infty (\Sigma_n)}
      \norma{u^{n+1/2}}_{\L\infty ([a,b])}\right).
  \end{align}

  Focus first on
  $\sum_{j=1}^{N-1} \modulo{u^{n+1/2}_{j+1} - u^{n+1/2}_j}$:
  by~\eqref{eq:scheme1} we obtain
  \begin{align*}
    &u^{n+1/2}_{j+1} - u^{n+1/2}_j
    \\
    = \
    & \uh{j+1} - \uh{j}
    \\
    & - \lambda \left[ F^n_{j+3/2} (\uh{j+1}, \uh{j+2}) - F^n_{j+1/2}
      (\uh{j}, \uh{j+1}) - F^n_{j+1/2} (\uh{j}, \uh{j+1}) +
      F^n_{j-1/2} (\uh{j-1}, \uh{j}) \right]
    \\
    & \pm \lambda \, F^n_{j+3/2} (\uh{j}, \uh{j+1}) \pm \, \lambda
      F^n_{j+1/2} (\uh{j-1}, \uh{j})
    \\
    = \
    & \uh{j+1} - \uh{j}
    \\
    & - \lambda \left[ F^n_{j+3/2} (\uh{j+1}, \uh{j+2}) - F^n_{j+1/2}
      (\uh{j}, \uh{j+1}) + F^n_{j+1/2} (\uh{j-1}, \uh{j}) -F^n_{j+3/2}
      (\uh{j}, \uh{j+1}) \right]
    \\
    & - \lambda \left[ F^n_{j+3/2} (\uh{j}, \uh{j+1}) -F^n_{j+1/2}
      (\uh{j-1}, \uh{j}) +F^n_{j-1/2} (\uh{j-1}, \uh{j}) - F^n_{j+1/2}
      (\uh{j}, \uh{j+1}) \right]
    \\
    = \
    & \mathcal{A}_j^n - \lambda \, \mathcal{B}_j^n,
  \end{align*}
  where we set
  \begin{align*}
    \mathcal{A}_j^n = \
    &\uh{j+1} - \uh{j}
    \\
    & - \lambda \left[ F^n_{j+3/2} (\uh{j+1}, \uh{j+2}) - F^n_{j+1/2}
      (\uh{j}, \uh{j+1}) + F^n_{j+1/2} (\uh{j-1}, \uh{j}) -F^n_{j+3/2}
      (\uh{j}, \uh{j+1}) \right],
    \\
    \mathcal{B}_j^n = \
    & F^n_{j+3/2} (\uh{j}, \uh{j+1}) -F^n_{j+1/2}
      (\uh{j-1}, \uh{j}) +F^n_{j-1/2} (\uh{j-1}, \uh{j}) - F^n_{j+1/2}
      (\uh{j}, \uh{j+1}).
  \end{align*}
  Rearrange $\mathcal{A}_j^n$ as follows:
  \begin{align*}
    \mathcal{A}_j^n = \
    & \uh{j+1} - \uh{j} - \lambda \,
      \frac{F^n_{j+3/2} (\uh{j+1}, \uh{j+2}) - F^n_{j+3/2} (\uh{j+1},
      \uh{j+1})}{\uh{j+2}- \uh{j+1}} \, \left(\uh{j+2}-
      \uh{j+1}\right)
    \\
    &- \lambda \, \frac{F^n_{j+3/2} (\uh{j+1}, \uh{j+1}) - F^n_{j+3/2}
      (\uh{j}, \uh{j+1})}{\uh{j+1}- \uh{j}} \, \left(\uh{j+1}-
      \uh{j}\right)
    \\
    & + \lambda \, \frac{F^n_{j+1/2} (\uh{j}, \uh{j+1}) - F^n_{j+1/2}
      (\uh{j}, \uh{j})}{\uh{j+1}- \uh{j}} \, \left(\uh{j+1}-
      \uh{j}\right)
    \\
    & + \lambda \, \frac{F^n_{j+1/2} (\uh{j}, \uh{j}) - F^n_{j+1/2}
      (\uh{j-1}, \uh{j})}{\uh{j}- \uh{j-1}} \, \left(\uh{j}-
      \uh{j-1}\right)
    \\
    = \
    & \delta_j^n \, (\uh{j} - \uh{j-1}) + \gamma_{j+1}^n \,
      (\uh{j+2} - \uh{j+1}) + (1 - \gamma_j^n - \delta_{j+1}^n)
      (\uh{j+1} - \uh{j}),
  \end{align*}
  where
  \begin{equation}
    \label{eq:deltajn}
    \delta_j^n =
    \begin{cases}
      \lambda \dfrac{F_{j+1/2}^n (\uh{j}, \uh{j})- F_{j+1/2}^n
        (\uh{j-1}, \uh{j})}{\uh{j}-\uh{j-1}} & \mbox{if } \uh{j} \neq
      \uh{j-1},
      \\
      0 & \mbox{if } \uh{j} = \uh{j-1},
    \end{cases}
  \end{equation}
  while $\gamma_j^n$ is as in~\eqref{eq:gammajn}. It can be proven
  that $\delta_j^n \in \left[0, 1/3 \right]$. Thus,
  \begin{align}
    \nonumber
    & \sum_{j=1}^{N-1} \modulo{\mathcal{A}_j^n}
    \\ \nonumber
    \leq \
    & \sum_{j=1}^{N-1} \modulo{\uh{j+1} - \uh{j}} +
      \sum_{j=0}^{N-2} \delta_{j+1}^n \, \modulo{\uh{j+1} - \uh{j}} -
      \sum_{j=1}^{N-1} \delta_{j+1}^n \, \modulo{\uh{j+1} - \uh{j}}
    \\
    \nonumber
    & + \sum_{j=2}^{N}\gamma_j^n \, \modulo{\uh{j+1} -
      \uh{j}} - \sum_{j=1}^{N-1} \gamma_j^n \, \modulo{\uh{j+1} -
      \uh{j}}
    \\
    \label{eq:Ajn} = \
    & \sum_{j=1}^{N-1}
      \modulo{\uh{j+1} - \uh{j}} + \delta_1^n \, \modulo{\uh{1} -
      \uh{a}} - \delta_N^n \, \modulo{\uh{N}- \uh{N-1}} + \gamma^n_N
      \, \modulo{\uh{b} - \uh{N}} - \gamma_1^n \, \modulo{\uh{2} -
      \uh{1}}.
  \end{align}
  Focus now on $\mathcal{B}_j^n$:
  \begin{align*}
    \mathcal{B}_j^n = \
    & \frac12 \left[ f (t^n,x_{j+3/2}, \uh{j} ) +
      f (t^n,x_{j+3/2}, \uh{j+1} ) \right.
    \\
    & - f (t^n,x_{j+1/2}, \uh{j}) - f (t^n,x_{j+1/2}, \uh{j+1} )
    \\
    & + f (t^n,x_{j-1/2}, \uh{j-1} ) + f (t^n,x_{j-1/2}, \uh{j})
    \\
    & \left.  - f (t^n,x_{j+1/2}, \uh{j-1} ) - f (t^n,x_{j+1/2},
      \uh{j} ) \right]
    \\
    = \
    & \frac12 \left( f (t^n,x_{j+3/2}, \uh{j+1} ) - f
      (t^n,x_{j+1/2}, \uh{j+1} ) \right)
    \\
    & + \frac12 \left( f (t^n,x_{j-1/2}, \uh{j-1} ) - f
      (t^n,x_{j+1/2}, \uh{j-1} ) \right)
    \\
    & + \frac12 \left( f (t^n,x_{j+3/2}, \uh{j} ) -2 \, f
      (t^n,x_{j+1/2}, \uh{j} ) + f (t^n,x_{j-1/2}, \uh{j} ) \right)
    \\
    = \
    & \frac{\dx}{2} \, \partial_x f (t^n,\tilde{x}_{j+1}, \uh{j+1}
      ) - \frac{\dx}2 \, \partial_x f (t^n,\tilde{x}_{j}, \uh{j-1})
    \\
    & + \frac{\dx}2 \left(
      \partial_x f (t^n,\bar{x}_{j+1}, \uh{j} ) - \partial_x f
      (t^n,\bar{x}_{j}, \uh{j} ) \right)
    \\
    = \
    & \frac{\dx}{2}\Bigl( (\tilde{x}_{j+1} - \tilde{x}_j)
      \, \partial_{xx}^2 f (t^n,\hat{x}_{j+1/2}, \uh{j+1} ) + (\uh{j+1}
      - \uh{j-1}) \, \partial_{xu}^2 f (t^n,\tilde{x}_{j},
      \tilde{u}^n_{j} )
    \\
    & \qquad + (\bar{x}_{j+1} - \bar{x}_j) \, \partial^2_{xx} f (t^n,
      \bar{x}_{j+1/2}, \uh{j})\Bigr),
  \end{align*}
  where
  \begin{align*}
    \tilde{x}_{j+1}, \bar{x}_{j+1} \in \
    & ]x_{j+1/2}, x_{j+3/2}[,
    &
      \tilde{x}_{j}, \bar{x}_{j} \in \
    & ]x_{j-1/2}, x_{j+1/2}[,
    \\
    \hat{x}_{j+1/2} \in \
    & ]\tilde{x}_{j}, \tilde{x}_{j+1}[,
    &
      \bar{x}_{j+1/2} \in \
    & ]\bar{x}_j, \bar{x}_{j+1}[,
    \\
    \tilde{u}^n_{j} \in \
    & \mathcal{I}\left(\uh{j-1}, \uh{j+1}\right).
  \end{align*}
  Notice that
  \begin{align*}
    \modulo{\tilde{x}_{j+1} - \tilde{x}_j} \leq \
    & 2 \, \dx,
    &
      \modulo{\bar{x}_{j+1} - \bar{x}_j} \leq \
    & 2 \, \dx.
  \end{align*}
  Hence,
  \begin{displaymath}
    \modulo{\mathcal{B}^n_j} \leq
    2 \, (\dx)^2 \,
    \norma{\partial_{xx}^2 f}_{\L\infty (\Sigma_n)} + \frac{\dx}2 \,
    \norma{\partial_{xu}^2 f}_{\L\infty (\Sigma_n)} \modulo{\uh{j+1} -
      \uh{j-1}},
  \end{displaymath}
  so that
  \begin{align}
    \nonumber \sum_{j=1}^{N-1} \lambda \, \modulo{\mathcal{B}_j^n}
    \leq \
    & \frac{\dt}{2} \norma{\partial_{xu}^2 f}_{\L\infty
      (\Sigma_n)} \sum_{j=1}^{N-1} \modulo{\uh{j+1} - \uh{j-1}} + 2 \,
      \dt \, (b-a) \, \norma{\partial_{xx}^2 f}_{\L\infty (\Sigma_n)}
    \\
    \label{eq:Bjn}
    \leq \
    & \dt \norma{\partial_{xu}^2 f}_{\L\infty (\Sigma_n)}
      \sum_{j=0}^{N-1} \modulo{\uh{j+1} - \uh{j}} + 2 \, \dt \, (b-a) \,
      \norma{\partial_{xx}^2 f}_{\L\infty (\Sigma_n)}.
  \end{align}
  Therefore, collecting together the estimates~\eqref{eq:Ajn}
  and~\eqref{eq:Bjn} we get:
  \begin{align}
    \nonumber
    & \sum_{j=1}^{N-1} \modulo{u_{j+1}^{n+1/2} - u_j^{n+1/2}}
    \\ \nonumber \leq \
    & \sum_{j=1}^{N-1}
      \modulo{\uh{j+1} - \uh{j}} + \delta_1^n \, \modulo{\uh{1} -
      \uh{a}} - \delta_N^n \, \modulo{\uh{N}- \uh{N-1}} + \gamma^n_N
      \, \modulo{\uh{b} - \uh{N}} - \gamma_1^n \, \modulo{\uh{2} -
      \uh{1}}
    \\
    \label{eq:6}
    & + \dt \norma{\partial_{xu}^2
      f}_{\L\infty (\Sigma_n)} \sum_{j=0}^{N-1} \modulo{\uh{j+1} -
      \uh{j}} + 2 \, \dt \, (b-a) \, \norma{\partial_{xx}^2
      f}_{\L\infty (\Sigma_n)}.
  \end{align}

  Focus now on the terms involving the boundary data
  in~\eqref{eq:4}. With the notation~\eqref{eq:betajn},
  \eqref{eq:gammajn} and~\eqref{eq:deltajn}, we observe that
  \begin{align*}
    &\beta_1^n \, (\uh{1}- \uh{a}) + \lambda \left( F_{3/2}^n (\uh{1},
      \uh{1}) - F_{1/2}^n (\uh{1}, \uh{1}) \right)
    \\
    = \
    & \lambda \left[ F_{1/2}^n (\uh{1}, \uh{1}) - F_{1/2}^n
      (\uh{a}, \uh{1}) + F_{3/2}^n (\uh{1}, \uh{1}) - F_{1/2}^n
      (\uh{1}, \uh{1}) \pm F_{3/2}^n (\uh{a}, \uh{1}) \right]
    \\
    = \
    & \delta_1^n \, (\uh{1} - \uh{a}) + \lambda \left( F_{3/2}^n
      (\uh{a}, \uh{1}) - F_{1/2}^n (\uh{a}, \uh{1}) \right).
  \end{align*}
  Hence, from the definition of the scheme~\eqref{eq:scheme1},we have
  \begin{align}
    \nonumber
    & u_1^{n+1/2} - u_a^{n+1}
    \\ \nonumber = \
    & (1 -
      \beta_1^n - \gamma_1^n) \uh{1} + \beta_1^n \, \uh{a} + \gamma_1^n
      \, \uh{2} - \lambda \left( F_{3/2}^n (\uh{1}, \uh{1}) - F_{1/2}^n
      (\uh{1}, \uh{1}) \right) - u_a^{n+1} \pm \uh{a}
    \\
    \nonumber =\
    & \gamma_1^n (\uh{2} - \uh{1}) + (1-\beta_1^n) (\uh{1}- \uh{a}) +
      (u_a^{n+1} - \uh{a}) - \lambda \left( F_{3/2}^n (\uh{1}, \uh{1}) -
      F_{1/2}^n (\uh{1}, \uh{1}) \right)
    \\
    \label{eq:2} = \
    & \gamma_1^n (\uh{2} - \uh{1}) + (1-\delta_1^n) (\uh{1}- \uh{a}) +
      (u_a^{n+1} - \uh{a}) - \lambda \left( F_{3/2}^n (\uh{a}, \uh{1}) -
      F_{1/2}^n (\uh{a}, \uh{1}) \right).
  \end{align}
  Since
  \begin{align*}
    & \lambda \left( F_{3/2}^n (\uh{a}, \uh{1}) - F_{1/2}^n (\uh{a},
      \uh{1}) \right)
    \\
    = \
    & \frac{\lambda}{2}\left[ f (t^n, x_{3/2}, \uh{a}) + f (t^n,
      x_{3/2}, \uh{1}) - f (t^n, x_{1/2}, \uh{a}) - f (t^n, x_{1/2},
      \uh{1}) \right]
    \\
    = \
    & \frac{\lambda}{2}\left[ \dx \, \partial_x f (t^n,
      \tilde{x}_{1}, \uh{a}) + \dx \, \partial_x f (t^n,
      \overline{x}_{1}, \uh{1}) \pm \dx \, \partial_x f (t^n,
      \tilde{x}_{1}, 0) \pm \dx \, \partial_x f (t^n,
      \overline{x}_{1}, 0) \right]
    \\
    = \
    & \frac{\dt}{2}\left(
          \partial_{xu}^2 f (t^n, \tilde{x}_1, \tilde{u}^n_{a}) \, \uh{a}
          + \partial_{xu}^2 f (t^n, \overline{x}_1, \overline{u}^n_{1}) \,
          \uh{1} + \partial_x f (t^n, \tilde{x}_{1}, 0) + \partial_x f
          (t^n, \overline{x}_{1}, 0) \right),
  \end{align*}
  where $\tilde{x}_1, \overline{x}_1 \in ]x_{1/2}, x_{3/2}[$,
  $\tilde{u}^n_{a} \in \mathcal{I} (0, \uh{a})$ and
  $\overline{u}_1^n \in \mathcal{I} (0,\uh{1})$, we conclude
  \begin{align*}
    & \lambda \, \modulo{ F_{3/2}^n (\uh{a}, \uh{1}) - F_{1/2}^n
      (\uh{a}, \uh{1}) }
    \\ \leq \
    & \dt \, \left(\frac12
      \,\norma{\partial_{xu}^2 f}_{\L\infty (\Sigma_n)} (|\uh{a}| +
      |\uh{1}|) + \norma{\partial_x f (\cdot, \cdot,0)}_{\L\infty
      ([0,t^n] \times[a,b])} \right).
  \end{align*}
  By the positivity of the coefficients involved in~\eqref{eq:2}, we
  obtain
  \begin{align}
    \label{eq:diff1a}
    \modulo{u_1^{n+1/2} - u_a^{n+1}} \leq \
    & \gamma_1^n \, |\uh{2} -
      \uh{1}| + (1-\delta_1^n) |\uh{1}- \uh{a}| + |u_a^{n+1} - \uh{a}|
    \\ \nonumber
    & +\dt \, \left(\frac12 \,\norma{\partial_{xu}^2
      f}_{\L\infty (\Sigma_n)} (|\uh{a}| + |\uh{1}|) +
      \norma{\partial_x f (\cdot, \cdot,0)}_{\L\infty ([0,t^n]
      \times[a,b])} \right).
  \end{align}
  Similarly as before, compute
  \begin{align*}
    &\beta_N^n (\uh{N}- \uh{N-1}) + \lambda \left( F_{N+1/2}^n
      (\uh{N}, \uh{N}) - F_{N-1/2}^n (\uh{N}, \uh{N}) \right)
    \\
    = \
    & \lambda \bigl[ F_{N-1/2}^n (\uh{N}, \uh{N}) - F_{N-1/2}^n
      (\uh{N-1}, \uh{N}) + F_{N+1/2}^n (\uh{N}, \uh{N})
    \\
    & \quad- F_{N-1/2}^n (\uh{N}, \uh{N}) \pm F_{N+1/2}^n (\uh{N-1},
      \uh{N}) \bigr]
    \\
    = \
    &\delta_N^n \, (\uh{N}- \uh{N-1}) + \lambda \left( F_{N+1/2}^n
      (\uh{N-1}, \uh{N}) - F_{N-1/2}^n (\uh{N-1}, \uh{N}) \right).
  \end{align*}
  Therefore, concerning the other boundary term, we have
  \begin{align}
    \nonumber u_b^{n+1} - u_N^{n+1/2} = \
    & u_b^{n+1} \pm \uh{b} - (1
      - \beta_N^n - \gamma_N^n) \, \uh{N} + \beta_N^n \, \uh{N-1} +
      \gamma_N^n \, \uh{b}
    \\ \nonumber
    & + \lambda \left( F_{N+1/2}^n
      (\uh{N}, \uh{N}) - F_{N-1/2}^n (\uh{N}, \uh{N}) \right)
    \\
    \nonumber =\
    & (u_b^{n+1} - \uh{b}) + (1 - \gamma_N^n) (\uh{b} -
      \uh{N}) + \beta_N^n \, (\uh{N}- \uh{N-1})
    \\ \nonumber
    &+ \lambda
      \left( F_{N+1/2}^n (\uh{N}, \uh{N}) - F_{N-1/2}^n (\uh{N}, \uh{N})
      \right)
    \\ \label{eq:3} = \
    & (u_b^{n+1} - \uh{b}) + (1 -
      \gamma_N^n) (\uh{b} - \uh{N}) + \delta_N^n \, (\uh{N}- \uh{N-1})
    \\ \label{eq:3a}
    & + \lambda \left( F_{N+1/2}^n (\uh{N-1}, \uh{N})
      - F_{N-1/2}^n (\uh{N-1}, \uh{N}) \right),
  \end{align}
  Observing that
  \begin{align*}
    & \lambda \left( F_{N+1/2}^n (\uh{N-1}, \uh{N}) - F_{N-1/2}^n
      (\uh{N-1}, \uh{N}) \right)
    \\
    = \
    & \frac{\lambda}{2}\left[ f (t^n, x_{N+1/2}, \uh{N-1}) + f
      (t^n, x_{N+1/2}, \uh{N}) - f (t^n, x_{N-1/2}, \uh{N-1}) - f
      (t^n, x_{N-1/2}, \uh{N}) \right]
    \\
    = \
    & \frac{\lambda}{2}\left[ \dx \, \partial_x f (t^n,
      \tilde{x}_{N}, \uh{N-1}) + \dx \, \partial_x f (t^n,
      \overline{x}_{N}, \uh{N}) \pm \dx \, \partial_x f (t^n,
      \tilde{x}_{N},0) \pm \dx \, \partial_x f (t^n, \overline{x}_{N},
      0) \right]
    \\
    = \
    & \frac{\dt}{2}\left(
      \partial_{xu}^2 f (t^n, \tilde{x}_N, \tilde{u}_{N-1}^n) \,
      \uh{N-1} + \partial_{xu}^2 f (t^n, \overline{x}_N,
      \overline{u}^n_N) \, \uh{N} + \partial_x f (t^n,
      \tilde{x}_{N},0) + \partial_x f (t^n, \overline{x}_{N}, 0)
      \right),
  \end{align*}
  where $\tilde{x}_N, \overline{x}_N \in ]x_{N-1/2}, x_{N+1/2}[$,
  $\tilde{u}^n_{N-1} \in \mathcal{I} (0,\uh{N-1})$ and
  $\overline{u}^n_{N} \in \mathcal{I} (0,\uh{N})$, we conclude
  \begin{align*}
    & \lambda \, \modulo{ F_{N+1/2}^n (\uh{N-1}, \uh{N}) - F_{N-1/2}^n
      (\uh{N-1}, \uh{N}) }
    \\
    \leq \
    & \dt \,\left(\frac12 \norma{\partial_{xu}^2 f}_{\L\infty
      (\Sigma_n)} (|\uh{N-1}| + |\uh{N}|) + \norma{\partial_x f
      (\cdot, \cdot, 0)}_{\L\infty ([0,t^n]\times[a,b])} \right).
  \end{align*}
  By the positivity of the coefficients involved
  in~\eqref{eq:3}--\eqref{eq:3a}, we obtain
  \begin{align}
    \label{eq:diffNb}
    \modulo{u_b^{n+1} - u_N^{n+1/2}} \leq \
    & |u_b^{n+1} -\uh{b}| + (1
      - \gamma_N^n) |\uh{b} - \uh{N}| + \delta_N^n |\uh{N}- \uh{N-1}|
    \\
    \nonumber
    & + \dt \, \left(\frac12 \, \norma{\partial_{xu}^2
      f}_{\L\infty (\Sigma_n)} (|\uh{N-1}| + |\uh{N}|) +
      \norma{\partial_x f (\cdot, \cdot, 0)}_{\L\infty
      ([0,t^n]\times[a,b])} \right).
  \end{align}

  Insert~\eqref{eq:6}, \eqref{eq:diff1a} and~\eqref{eq:diffNb}
  into~\eqref{eq:4}:
  \begin{align*}
    & \sum_{j=0}^N \modulo{u_{j+1}^{n+1} -u_j^{n+1}}
    \\
    \leq \
    & e^{\dt \, \norma{\partial_u g}_{\L\infty (\Sigma_n)}}
      \biggl[ \gamma_1^n |\uh{2} - \uh{1}| + (1-\delta_1^n) |\uh{1}-
      \uh{a}| + |u_a^{n+1} - \uh{a}|
    \\
    &+ \dt \, \left(\frac12 \, \norma{\partial_{xu}^2 f}_{\L\infty
      (\Sigma_n)} (|\uh{a}| + |\uh{1}|) + \norma{\partial_x f
      (\cdot, \cdot,0)}_{\L\infty ([0,t^n] \times[a,b])} \right)
    \\
    & + \sum_{j=1}^{N-1} \modulo{\uh{j+1} - \uh{j}} + \delta_1^n \,
      \modulo{\uh{1} - \uh{a}} - \delta_N^n \, \modulo{\uh{N}- \uh{N-1}}
      + \gamma^n_N \, \modulo{\uh{b} - \uh{N}} - \gamma_1^n \,
      \modulo{\uh{2} - \uh{1}}
    \\
    & + \dt \norma{\partial_{xu}^2 f}_{\L\infty (\Sigma_n)}
      \sum_{j=0}^{N-1} \modulo{\uh{j+1} - \uh{j}} + 2 \, \dt \, (b-a) \,
      \norma{\partial_{xx}^2 f}_{\L\infty (\Sigma_n)}
    \\
    & +|u_b^{n+1} - \uh{b}| + (1 - \gamma_N^n) |\uh{b} - \uh{N}| +
      \delta_N^n |\uh{N}- \uh{N-1}|
    \\
    & + \dt \, \left( \frac12 \,\norma{\partial_{xu}^2 f}_{\L\infty
      (\Sigma_n)} (|\uh{N-1}| + |\uh{N}|) + \norma{\partial_x f
      (\cdot, \cdot, 0)}_{\L\infty ([0,t^n]\times[a,b])} \right)
      \biggr]
    \\
    & + \dt \, (b-a) \norma{\partial_xg}_{\L\infty (\Sigma_n)}\!  + \!
      2 \dt \norma{g (\cdot, \cdot, 0)}_{\L\infty ([0,t^n]\times
      [a,b])}\!  + \! 2 \dt \, \norma{\partial_u g}_{\L\infty
      (\Sigma_n)} \norma{u^{n+1/2}}_{\L\infty ([a,b])}
    \\
    \leq \
    & e^{\dt \, \norma{\partial_u g}_{\L\infty (\Sigma_n)}}
      \biggl[ |u_a^{n+1} - \uh{a}| +|u_b^{n+1} - \uh{b}| + \left( 1+ \dt
      \, \norma{\partial_{xu}^2 f}_{\L\infty (\Sigma_n)} \right)
      \sum_{j=0}^{N} \modulo{\uh{j+1} - \uh{j}}
    \\
    & + \mathcal{K}_1 (t^n) \,\dt + \frac12 \,\left( 3 \, \mathcal{U}
      (t^n) + \norma{u_a}_{\L\infty ([0,t^n])}\right)
      \norma{\partial_{xu}^2 f}_{\L\infty (\Sigma_n)} \, \dt \biggr]
    \\
    & + \dt \, (b-a) \norma{\partial_xg}_{\L\infty (\Sigma_n)} + 2 \dt
      \norma{g (\cdot, \cdot, 0)}_{\L\infty ([0,t^n]\times [a,b])} + 2
      \, \dt \, \norma{\partial_u g}_{\L\infty (\Sigma_n)} \,
      \mathcal{U} (t^n),
  \end{align*}
  where
  \begin{align}
    \label{eq:K1}
    \mathcal{K}_1 (t) = \
    & 2 \left( \norma{\partial_x f (\cdot,
      \cdot, 0)}_{\L\infty ([0,t^n]\times[a,b])} + (b-a) \,
      \norma{\partial_{xx}^2 f}_{\L\infty (\Sigma_n)} \right).
  \end{align}
  Setting
  \begin{equation}
    \label{eq:K2}
    \begin{aligned}
      \mathcal{K}_2 (t) = \
      & 2 \, \mathcal{C}_1 (t) +(b-a) \left(2 \,
        \norma{\partial_{xx}^2 f}_{\L\infty (\Sigma_t)} +
        \norma{\partial_xg}_{\L\infty (\Sigma_t)}\right)
      \\
      & + \frac12 \,\left( 3 \, \mathcal{U} (t) +
        \norma{u_a}_{\L\infty ([0,t])}\right) \norma{\partial_{xu}^2
        f}_{\L\infty (\Sigma_t)} \, \dt + 2 \, \norma{\partial_u
        g}_{\L\infty (\Sigma_t)} \, \mathcal{U} (t),
    \end{aligned}
  \end{equation}
  with $\mathcal{U}(t)$ as in~\eqref{eq:Ut},
  $\Sigma_t = [0,t] \times [a,b] \times [-\mathcal{U} (t) ,
  \mathcal{U} (t)]$ and $\mathcal{C}_1 (t)$ as in~\eqref{eq:c1}, we
  deduce by a standard iterative procedure the following estimate
  \begin{displaymath}
    \sum_{j=0}^N\! \modulo{u_{j+1}^{n}  -u_j^{n}}\! \leq\!
    e^{\mathcal{C}_2 (t^n) \, t^n} \!\! \left(
      \sum_{j=0}^N \modulo{u_{j+1}^{0}  -u_j^{0}}\!
      +\! \! \sum_{m=1}^n \modulo{u_{a}^{m}  -u_a^{m-1}}\!
      +\!  \sum_{m=1}^n \modulo{u_{b}^{m}  -u_b^{m-1}}\!
      +\! t^n \, \mathcal{K}_2 (t^n)\!\right),
  \end{displaymath}
  where actually the norms appearing in $\mathcal{C}_2 (t)$
  in~\eqref{eq:c2} can now be computed on $\Sigma_t$ instead of on
  $[0,t] \times [a,b] \times \reali$.
\end{proof}

\begin{corollary}\label{cor:BVxt} {\bf ($\BV$ estimate in space and time)}
  Let~\ref{f}, \ref{g}, \ref{ID} and~\eqref{eq:CFL} hold. Then,
  for $n$ between $1$ and $N_T$, the following estimate holds
  \begin{equation}
    \label{eq:BVxt}
    \sum_{m=0}^{n-1} \sum_{j=0}^{N}
    \dt \, \modulo{u^{m}_{j+1}-u^{m}_{j}}
    +
    \sum_{m=0}^{n-1} \sum_{j=0}^{N+1}
    \dx \, \modulo{u_j^{m+1} - u_j^m}
    \leq
    \mathcal{C}_{xt} (t^n),
  \end{equation}
  where $\mathcal{C}_{xt} (t^n)$ is given by~\eqref{eq:Cxt}.
\end{corollary}

\begin{proof}
  By Proposition~\ref{prop:BV}, we have
  \begin{equation}
    \label{eq:7}
    \sum_{m=0}^{n-1} \sum_{j=0}^{N}
    \dt \, \modulo{u^{m}_{j+1}-u^{m}_{j}}
    \leq n \, \dt \, \mathcal{C}_x (t^n).
  \end{equation}
  By the definition of the
  scheme~\eqref{eq:scheme1}--\eqref{eq:scheme2}, for
  $m \in \left\{0, \ldots, n-1\right\}$ and
  $j\in\left\{1, \ldots, N \right\}$, we have
  \begin{equation}
    \label{eq:32}
    \modulo{u_j^{m+1} - u_j^m} \leq  \modulo{u_j^{m+1} - u_j^{m+1/2}} + \modulo{u_j^{m+1/2} - u_j^m}.
  \end{equation}
  In particular, the first term on the r.h.s.~of~\eqref{eq:32} can be
  estimated as follows
  \begin{equation}
    \label{eq:33}
    \modulo{u_j^{m+1} - u_j^{m+1/2}}
    \leq
    \dt \, \norma{g}_{\L\infty ([0,t^m] \times [a,b] \times [-\mathcal{U} (t^m), \, \mathcal{U} (t^m)])},
  \end{equation}
  and the norm of $g$ appearing above is bounded thanks to~\ref{g}.
  Concerning the second term on the r.h.s.~of~\eqref{eq:32},
  by~\eqref{eq:numFl} and~\eqref{eq:scheme1}, we obtain
  \begin{align*}
    & \modulo{u_j^{m+1/2} - u_j^m}
    \\ \leq \
    & \frac{\lambda \,
      \alpha}{2} \left( \modulo{u_{j+1}^m - u_j^m} + \modulo{u_j^m -
      u_{j-1}^m} \right)
    \\
    & + \frac{\lambda}{2}\left| f (t^m, x_{j+1/2}, u_j^m) + f (t^m,
      x_{j+1/2}, u_{j+1}^m) - f (t^m, x_{j-1/2}, u_{j-1}^m) - f (t^m,
      x_{j-1/2}, u_j^m) \right|
    \\
    \leq \
    & \frac{\lambda \, \alpha}{2} \left( \modulo{u_{j+1}^m -
      u_j^m} + \modulo{u_j^m - u_{j-1}^m} \right)
    \\
    & + \frac\lambda2 \left[ \modulo{\partial_x f (t^m, \tilde{x}_{j},
      u_{j}^m)} \dx + \modulo{\partial_u f (t^m, x_{j-1/2},
      \tilde{u}_{j-1/2}^m)}\modulo{u_j^m - u_{j-1}^m} \right.
    \\
    & \left. \qquad + \modulo{\partial_x f (t^m, \overline{x}_{j},
      u_{j+1}^m)} \dx + \modulo{\partial_u f (t^m, x_{j-1/2},
      \tilde{u}_{j+1/2}^m)}\modulo{u_{j+1}^m - u_{j}^m} \right]
    \\
    \leq \
    & \frac\lambda2 \, \left(\alpha + \norma{\partial_u
      f}_{\L\infty ([0,t^m] \times [a,b] \times [-\mathcal{U} (t^m),
      \, \mathcal{U} (t^m)])}\right) \left( \modulo{u_{j+1}^m -
      u_j^m} + \modulo{u_j^m - u_{j-1}^m} \right)
    \\
    & + \dt \, \norma{\partial_x f}_{\L\infty ([0,t^m] \times [a,b]
      \times [-\mathcal{U} (t^m), \, \mathcal{U} (t^m)])},
  \end{align*}
  where $\tilde{x}_{j}, \overline{x}_j \in \, ]x_{j-1/2}, x_{j+1/2}[$,
  $\tilde{u}_{j-1/2}^m \in \mathcal{I} (u_{j-1}^m, u_j^m)$ and
  $\tilde{u}_{j+1/2}^m \in \mathcal{I} (u_{j}^m, u_{j+1}^m)$. Remark
  that, by~\ref{f}, the norm of $\partial_x f$ appearing above is
  bounded. Therefore, by Proposition~\ref{prop:BV},
  \begin{align}
    \nonumber \sum_{j=1}^N \dx \, \modulo{u_j^{m+1/2} - u_j^m} \leq \
    & \dt \, \left(\alpha + \norma{\partial_u f}_{\L\infty ([0,t^m]
      \times [a,b] \times [-\mathcal{U} (t^m), \, \mathcal{U}
      (t^m)])}\right) \sum_{j=0}^N \modulo{u_{j+1}^m - u_j^m} \\
    \nonumber
    & + \dt \, (b-a) \,\norma{\partial_x f}_{\L\infty
      ([0,t^m] \times [a,b] \times [-\mathcal{U} (t^m), \, \mathcal{U}
      (t^m)])}
    \\
    \label{eq:34} = \
    & \dt \, \mathcal{C}_t (t^m),
  \end{align}
  where we set
  \begin{equation}
    \label{eq:Ct}
    \mathcal{C}_t (\tau) =
    \left(\alpha +  \norma{\partial_u f}_{\L\infty ([0,\tau] \times [a,b] \times [-\mathcal{U} (\tau), \, \mathcal{U} (\tau)])}\right)
    \mathcal{C}_x (\tau)
    +  (b-a) \, \norma{\partial_x f}_{\L\infty ([0,\tau] \times [a,b] \times [-\mathcal{U} (\tau), \, \mathcal{U} (\tau)])}.
  \end{equation}
  Exploiting~\eqref{eq:32}, \eqref{eq:33} and~\eqref{eq:34}, we obtain
  \begin{align*}
    \sum_{j=0}^{N+1} \dx \, \modulo{u_j^{m+1} - u_j^m} = \
    & = \dx \,
      \modulo{u_a^{m+1} - u_a^m} + \dx \, \modulo{u_b^{m+1} - u_b^m} +
      \sum_{j=1}^{N} \dx \, \modulo{u_j^{m+1} - u_j^m}
    \\
    \leq \
    & \dx \, \modulo{u_a^{m+1} - u_a^m} + \dx \,
      \modulo{u_b^{m+1} - u_b^m}
    \\
    & + \dt \, \left( \mathcal{C}_t (t^m) + \norma{g}_{\L\infty
      ([0,t^m] \times [a,b] \times [-\mathcal{U} (t^m), \,
      \mathcal{U} (t^m)])} \right),
  \end{align*}
  which, summed over $m=0, \ldots, n-1$, yields
  \begin{equation}
    \label{eq:8}
    \begin{aligned}
      \sum_{m=0}^{n-1} \sum_{j=0}^{N+1} \dx \, \modulo{u_j^{m+1} -
        u_j^m} \leq \ & \dx \sum_{m=0}^{n-1} \left( \modulo{u_a^{m+1}
          - u_a^m} + \modulo{u_b^{m+1} - u_b^m} \right)
      \\
      & + n \, \dt \, \left( \mathcal{C}_t (t^n) + \norma{g}_{\L\infty
          ([0,t^n] \times [a,b] \times [-\mathcal{U} (t^n), \,
          \mathcal{U} (t^n)])} \right).
    \end{aligned}
  \end{equation}

  Summing~\eqref{eq:7} and~\eqref{eq:8} we obtain the desired
  estimate~\eqref{eq:BVxt}, with
  \begin{align}
    \nonumber \mathcal{C}_{xt} (t^n) = \
    & n \, \dt \, ( 1 + \alpha +
      \norma{\partial_u f}_{\L\infty ([0,t^n] \times [a,b] \times
      [-\mathcal{U} (t^n), \, \mathcal{U} (t^n)])}) \, \mathcal{C}_x
      (t^n)
    \\ \nonumber
    &+n \, \dt \left( (b-a) \, \norma{\partial_x
      f}_{\L\infty ([0,t^n] \times [a,b] \times [-\mathcal{U} (t^n),
      \, \mathcal{U} (t^n)])} + \norma{g}_{\L\infty ([0,t^n] \times
      [a,b] \times [-\mathcal{U} (t^n), \, \mathcal{U} (t^n)])}
      \right)
    \\ \label{eq:Cxt}
    & + \dx \sum_{m=0}^{n-1} \left(
      \modulo{u_a^{m+1} - u_a^m} + \modulo{u_b^{m+1} - u_b^m} \right),
  \end{align}
  concluding the proof.
\end{proof}

\subsection{Discrete entropy inequality}
\label{sec:Eineq}

We introduce the following notation: for $j=1, \ldots, N$,
$n = 0, \ldots, N_T-1$, $k \in \reali$,
\begin{align*}
  F_{j+1/2}^n (u,v) = \
  & \frac{1}{2} \left[ f (t^n, x_{j+1/2}, u) +
    f(t^n, x_{j+1/2}, v) \right] - \frac{\alpha}{2} \left( v - u \right),
  \\
  H^n_j (u,v,z) = \
  & v -\lambda \, \left( F_{j+1/2}^n (v,z) - F_{j-1/2}^n (u,v)\right),
  \\
  G^{n,k}_{j+1/2} (u,v) = \
  & F_{j+1/2}^n (u \wedge k,v \wedge k) - F_{j+1/2}^n (k,k),
  \\
  L^{n,k}_{j+1/2} (u,v) = \
  & F_{j+1/2}^n (k,k) - F_{j+1/2}^n (u \vee k,v \vee k).
\end{align*}
Observe that, due to the definition of the scheme,
$u_j^{n+1/2} = H^n_j (\uh{j-1}, \uh{j}, \uh{j+1})$.  Notice moreover
that the following equivalences hold true: $(s-k)^+ = s \wedge k - k$
and $(s-k)^- = k - s \vee k$.

\begin{lemma}
  \label{lem:Eineq}
  Let~\ref{f}, \ref{g}, \ref{ID} and~\eqref{eq:CFL} hold. Then the
  approximate solution $u_\Delta$ in~\eqref{eq:udelta}, defined
  through Algorithm~\ref{alg:1}, satisfies the following discrete
  entropy inequalities: for $j=1, \ldots, N$, $n = 0, \ldots, N_T-1$,
  $k \in \reali$,
  \begin{align}
    \nonumber (u_j^{n+1} - k)^+ - (\uh{j} -k)^+ + \lambda \, \left(
    G^{n,k}_{j+1/2} (\uh{j}, \uh{j+1}) - G^{n,k}_{j-1/2} (\uh{j-1},
    \uh{j}) \right)
    &
    \\
    \label{eq:Eineq+}
    + \lambda \, \sgn^+
    (u_j^{n+1} - k) \, \left( f (t^n, x_{j+1/2}, k) - f (t^n,
    x_{j-1/2}, k) \right)
    &
    \\ \nonumber
    - \dt \, \sgn^+ (u_j^{n+1}
    - k) \, g \!\left(t^n,x_{j}, u^{n+1/2}_j\right)
    &
      \leq 0,
  \end{align}
  and
  \begin{align}
    \nonumber (u_j^{n+1} - k)^- - (\uh{j} -k)^- + \lambda \, \left(
    L^{n,k}_{j+1/2} (\uh{j}, \uh{j+1}) - L^{n,k}_{j-1/2} (\uh{j-1},
    \uh{j}) \right)
    &
    \\ \label{eq:Eineq-}
    + \lambda \, \sgn^-
    (u_j^{n+1} - k) \, \left( f (t^n, x_{j+1/2}, k) - f (t^n,
    x_{j-1/2}, k) \right)
    &
    \\ \nonumber
    -\dt \, \sgn^- (u_j^{n+1} -
    k) \, g\! \left(t^n,x_{j}, u^{n+1/2}_j\right)
    & \leq 0.
  \end{align}
\end{lemma}

\begin{proof}
  Consider the map $(u,v,z) \mapsto H^n_j (u,v,z)$. By the CFL
  condition~\eqref{eq:CFL}, it holds
  \begin{align*}
    \frac{\partial H^n_j}{\partial u} (u,v,z) = \
    & \frac{\lambda}{2}
      \left(
      \partial_u f (t^n, x_{j-1/2}, u) + \alpha \right) \geq 0,
    \\
    \frac{\partial H^n_j}{\partial v} (u,v,z) = \
    & 1 - \lambda \,
      \alpha - \frac{\lambda}{2} \left(
      \partial_u f (t^n, x_{j+1/2}, v) - \partial_u f (t^n, x_{j-1/2},
      v) \right) \geq 0,
    \\
    \frac{\partial H^n_j}{\partial z} (u,v,z) = \
    &\frac{\lambda}{2}
      \left( \alpha - \partial_u f (t^n, x_{j+1/2}, z) \right) \geq 0.
  \end{align*}
  Notice that
  \begin{displaymath}
    H^n_j (k,k,k) = k - \lambda \left( f (t^n, x_{j+1/2}, k) - f (t^n, x_{j-1/2}, k)\right).
  \end{displaymath}
  By the monotonicity properties obtained above, we have
  \begin{align}
    \nonumber
    & H^n_j (\uh{j-1} \wedge k, \uh{j} \wedge k, \uh{j+1}
      \wedge k) - H^n_j (k,k,k)
    \\ \nonumber \geq \
    & H^n_j (\uh{j-1},
      \uh{j}, \uh{j+1}) \wedge H^n_j (k,k,k) - H^n_j (k,k,k)
    \\
    \nonumber = \
    & \left(H^n_j (\uh{j-1}, \uh{j}, \uh{j+1}) - H^n_j
      (k,k,k)\right) ^+
    \\ \nonumber =\
    & \left( u_j^{n+1/2} - k +
      \lambda \left( f (t^n, x_{j+1/2}, k) - f (t^n, x_{j-1/2},
      k)\right) \right)^+
    \\ \label{eq:serve1} = \
    & \left(
      u_j^{n+1} - k - \dt \, g \!  \left(t^n,x_j,u_j^{n+1/2}\right) +
      \lambda \left( f (t^n, x_{j+1/2}, k) - f (t^n, x_{j-1/2},
      k)\right) \right)^+.
  \end{align}
  Moreover, we also have
  \begin{align}
    \nonumber
    & H^n_j (\uh{j-1} \wedge k, \uh{j} \wedge k, \uh{j+1}
      \wedge k) - H^n_j (k,k,k)
    \\ \nonumber = \
    & (\uh{j} \wedge k) - k
    \\ \nonumber
    & - \lambda \left[ F^n_{j+1/2} (\uh{j} \wedge k,
      \uh{j+1} \wedge k) - F^n_{j-1/2} (\uh{j-1} \wedge k, \uh{j}
      \wedge k) - F^n_{j+1/2} (k,k) + F^n_{j-1/2}(k,k) \right]
    \\ \label{eq:serve2} = \
    & (\uh{j} - k)^+ - \lambda \left(
      G^{n,k}_{j+1/2} (\uh{j}, \uh{j+1}) - G^{n,k}_{j-1/2} (\uh{j-1},
      \uh{j})\right).
  \end{align}
  Hence, by~\eqref{eq:serve1} and~\eqref{eq:serve2},
  \begin{align*}
    & (\uh{j} - k)^+ - \lambda \left( G^{n,k}_{j+1/2} (\uh{j},
      \uh{j+1}) - G^{n,k}_{j-1/2} (\uh{j-1}, \uh{j})\right)
    \\
    \geq \
    & \left( u_j^{n+1} - k - \dt \, g \!
      \left(t^n,x_j,u_j^{n+1/2}\right) + \lambda \left( f (t^n,
      x_{j+1/2}, k) - f (t^n, x_{j-1/2}, k)\right) \right)^+
    \\
    = \
    & \sgn^+\left( u_j^{n+1} - k - \dt \, g \!
      \left(t^n,x_j,u_j^{n+1/2}\right) + \lambda \left( f (t^n,
      x_{j+1/2}, k) - f (t^n, x_{j-1/2}, k)\right) \right)
    \\
    & \times \left( u_j^{n+1} - k - \dt \, g \!
      \left(t^n,x_j,u_j^{n+1/2}\right) + \lambda \left( f (t^n,
      x_{j+1/2}, k) - f (t^n, x_{j-1/2}, k)\right)\right)
    \\
    \geq \
    &\left(u_j^{n+1} - k \right) ^+ - \dt \, \sgn^+\left(
      u_j^{n+1} - k \right) \, g \! \left(t^n,x_j,u_j^{n+1/2}\right)
    \\
    & + \lambda \, \sgn^+\left( u_j^{n+1} - k \right) \left(f (t^n,
      x_{j+1/2}, k) - f (t^n, x_{j-1/2}, k) \right),
  \end{align*}
  proving~\eqref{eq:Eineq+}, while~\eqref{eq:Eineq-} is proven in an
  entirely similar way.
\end{proof}

\subsection{Convergence towards an entropy weak solution}
\label{sec:conv}

The uniform $\L\infty$-bound provided by Lemma~\ref{lem:linf} and the
total variation estimate in Corollary~\ref{cor:BVxt} allow to apply
Helly's compactness theorem, ensuring the existence of a subsequence
of $u_\Delta$, still denoted by $u_\Delta$, converging in $\L1$ to a
function $u \in \L\infty ([0,T[ \times ]a,b[)$, for all $T>0$. We need
to prove that this limit function is indeed an MV--solution to the
IBVP~\eqref{eq:original}, in the sense of Definition~\ref{def:MV}.

\begin{lemma}
  \label{lem:entropyLimit}
  Let~\ref{f}, \ref{g}, \ref{ID} and~\eqref{eq:CFL} hold. Then the
  piecewise constant approximate solutions $u_\Delta$ resulting from
  Algorithm~\ref{alg:1} converge, as $\dx \to 0$, towards an
  MV--solution of the IBVP~\eqref{eq:original}.
\end{lemma}

\begin{proof}
  We consider the discrete entropy inequality~\eqref{eq:Eineq+}, for
  the positive semi-entropy, and we follow~\cite{DFG}, see
  also~\cite{Vovelle2002}. The proof for the negative semi-entropy is
  done analogously. Add and subtract
  $G^{n,k}_{j+1/2} (\uh{j}, \uh{j})$ in~\eqref{eq:Eineq+} and
  rearrange it as follows
  \begin{align*}
    0 \geq \
    & (u_j^{n+1} - k)^+ - (\uh{j} -k)^+ + \lambda \, \left(
      G^{n,k}_{j+1/2} (\uh{j}, \uh{j+1}) - G^{n,k}_{j+1/2} (\uh{j},
      \uh{j})\right)
    \\
    &+ \lambda \left( G^{n,k}_{j+1/2} (\uh{j}, \uh{j}) -
      G^{n,k}_{j-1/2} (\uh{j-1}, \uh{j}) \right) - \dt \, \sgn^+
      (u_j^{n+1} - k) \, g \!\left(t^n,x_{j}, u^{n+1/2}_j\right)
    \\
    & + \lambda \, \sgn^+ (u_j^{n+1} - k) \, \left( f (t^n, x_{j+1/2},
      k, \uh{j+1/2}) - f (t^n, x_{j-1/2}, k, \uh{j-1/2}) \right).
  \end{align*}
  Let $\phi \in \Cc1 ([0,T[ \times [a,b]; \reali^+)$ for some $T > 0$,
  multiply the inequality above by $\dx \, \phi (t^n, x_j)$ and sum
  over $j=1, \ldots, N$ and $n=0, \ldots, N_T-1$, so to get
  \begin{align}
    \label{eq:9}
    0 \geq \
    & \dx \sum_{n=0}^{N_T-1} \sum_{j=1}^N \left[ (u_j^{n+1} -
      k )^+ - (\uh{j} - k)^+ \right] \, \phi (t^n, x_j)
    \\
    \label{eq:9a}
    & + \dt \sum_{n=0}^{N_T-1} \sum_{j=1}^N \left[
      \left(G^{n,k}_{j+1/2} (\uh{j}, \uh{j+1}) - G^{n,k}_{j+1/2}
      (\uh{j}, \uh{j})\right) \right.
    \\
    \label{eq:9b}
    & \qquad \left.  - \left(G^{n,k}_{j-1/2} (\uh{j-1}, \uh{j}) -
      G^{n,k}_{j+1/2} (\uh{j}, \uh{j}) \right) \right] \, \phi (t^n,
      x_j) \\ \label{eq:9d}
    & - \dx \, \dt \sum_{n=0}^{N_T-1}
      \sum_{j=1}^N \sgn^+ (u_j^{n+1} - k) \, g \!\left(t^n,x_{j},
      u^{n+1/2}_j\right) \, \phi (t^n, x_j)
    \\
    \label{eq:9c}
    & + \dt \!\sum_{n=0}^{N_T-1} \!\sum_{j=1}^N \sgn^+ (u_j^{n+1} - k)
      \!\left( f (t^n, x_{j+1/2}, k, \uh{j+1/2})\! - \!f (t^n,
      x_{j-1/2}, k, \uh{j-1/2}) \right) \phi (t^n, x_j).
  \end{align}
  Consider each term separately. Summing by parts, we obtain
  \begin{align*}
    [\eqref{eq:9}] = \
    & - \dx \sum_{j=1}^N (u_j^0 - k)^+ \, \phi
      (0,x_j) - \dx \, \dt \sum_{n=1}^{N_T-1} \sum_{j=1}^N (\uh{j} -
      k)^+ \frac{\phi (t^n, x_j) - \phi (t^{n-1}, x_j)}{\dt}
    \\
    \underset{\dx \to 0^+}{\longrightarrow}
    & - \int_a^b (u_o (x) -
      k)^+ \, \phi (0,x) \d{x} - \int_0^{T} \int_a^b (u (t,x) - k)^+
      \, \partial_t\phi (t,x) \d{x}\d{t},
    \\[4pt]
    [\eqref{eq:9d}] = \
    & - \dx \, \dt \sum_{n=0}^{N_T-1} \sum_{j=1}^N
      \sgn^+ (u_j^{n+1} - k) \, g \!\left(t^n,x_{j}, u^{n+1/2}_j\right)
      \, \phi (t^n, x_j)
    \\
    \underset{\dx \to 0^+}{\longrightarrow}
    & - \int_0^{T} \int_a^b
      \sgn^+ (u (t,x) - k) \, g \left(t,x,u (t,x)\right) \, \phi (t,x)
      \d{x}\d{t},
    \\[4pt]
    [\eqref{eq:9c}] = \
    & \dx \, \dt \sum_{n=0}^{N_T-1} \sum_{j=1}^N
      \sgn^+ (u_j^{n+1} - k) \, \frac{f (t^n, x_{j+1/2}, k) - f (t^n,
      x_{j-1/2}, k)}{\dx} \, \phi (t^n, x_j)
    \\
    \underset{\dx \to 0^+}{\longrightarrow}
    & \int_0^{T} \int_a^b
      \sgn^+ (u (t,x) - k) \,
      \partial_x f (t, x, k) \, \phi (t,x) \d{x} \d{t},
  \end{align*}
  by the Dominated Convergence
  Theorem. Concerning~\eqref{eq:9a}--\eqref{eq:9b}, we get
  \begin{align*}
    [\eqref{eq:9a}-\eqref{eq:9b}] = \
    & \dt \sum_{n=0}^{N_T-1}
      \sum_{j=1}^N \left(G^{n,k}_{j+1/2} (\uh{j}, \uh{j+1}) -
      G^{n,k}_{j+1/2} (\uh{j}, \uh{j})\right) \, \phi (t^n, x_j)
    \\
    & - \dt \sum_{n=0}^{N_T-1} \sum_{j=0}^{N-1} \left(G^{n,k}_{j+1/2}
      (\uh{j}, \uh{j+1}) - G^{n,k}_{j+3/2} (\uh{j+1}, \uh{j+1})
      \right) \, \phi (t^n, x_{j+1})
    \\
    = \
    & \dt \sum_{n=0}^{N_T-1} \sum_{j=1}^{N-1} \left[
      \left(G^{n,k}_{j+1/2} (\uh{j}, \uh{j+1}) - G^{n,k}_{j+1/2}
      (\uh{j}, \uh{j})\right) \, \phi (t^n, x_j) \right.
    \\
    & \qquad \left.  - \left(G^{n,k}_{j+1/2} (\uh{j}, \uh{j+1}) -
      G^{n,k}_{j+3/2} (\uh{j+1}, \uh{j+1}) \right) \, \phi (t^n,
      x_{j+1}) \right]
    \\
    & + \dt \sum_{n=0}^{N_T-1} \left[ \left(G^{n,k}_{N+1/2} (\uh{N},
      \uh{b}) - G^{n,k}_{N+1/2} (\uh{N}, \uh{N})\right) \, \phi
      (t^n, x_N) \right.
    \\
    & \qquad \left.  - \left(G^{n,k}_{1/2} (\uh{a}, \uh{1}) -
      G^{n,k}_{3/2} (\uh{1}, \uh{1}) \right) \, \phi (t^n, x_{1})
      \right]
    \\
    = \
    & T^{int} + T^b = T,
  \end{align*}
  where we set
  \begin{align*}
    T^{int} = \
    & \dt \sum_{n=0}^{N_T-1} \sum_{j=1}^{N-1} \left[
      \left(G^{n,k}_{j+1/2} (\uh{j}, \uh{j+1}) - G^{n,k}_{j+1/2}
      (\uh{j}, \uh{j})\right) \, \phi (t^n, x_j) \right.
    \\
    & \qquad \left.  - \left(G^{n,k}_{j+1/2} (\uh{j}, \uh{j+1}) -
      G^{n,k}_{j+3/2} (\uh{j+1}, \uh{j+1}) \right) \, \phi (t^n,
      x_{j+1}) \right],
    \\
    T^b = \
    & \dt \sum_{n=0}^{N_T-1} \left[ \left(G^{n,k}_{N+1/2}
      (\uh{N}, \uh{b}) - G^{n,k}_{N+1/2} (\uh{N}, \uh{N})\right) \,
      \phi (t^n, x_N) \right.
    \\
    & \qquad \left.  - \left(G^{n,k}_{1/2} (\uh{a}, \uh{1}) -
      G^{n,k}_{3/2} (\uh{1}, \uh{1}) \right) \, \phi (t^n, x_{1})
      \right].
  \end{align*}
  Define
  \begin{equation}
    \label{eq:S}
    \begin{aligned}
      S = \
      & - \dx \, \dt \sum_{n=0}^{N_T-1} \sum_{j=1}^N
      G^{n,k}_{j+1/2} (\uh{j}, \uh{j}) \, \frac{\phi (t^n,x_{j+1}) -
        \phi (t^n,x_j)}{\dx}
      \\
      &- \alpha \, \dt \sum_{n=0}^{N_T-1} \left( (\uh{a} -k)^+ \, \phi
        (t^n,a) + (\uh{b}-k)^+ \, \phi (t^n,b) \right).
    \end{aligned}
  \end{equation}
  Observe that
  \begin{align*}
    G^{n,k}_{j+1/2} (\uh{j}, \uh{j}) = \
    & F^n_{j+1/2} (\uh{j} \wedge
      k, \uh{j} \wedge k) - F^n_{j+1/2} (k,k)
    \\
    = \
    & f (t^n, x_{j+1/2}, \uh{j}\wedge k) - f (t^n, x_{j+1/2}, k)
    \\
    = \
    & \sgn^+ (\uh{j} - k ) \left(f (t^n, x_{j+1/2}, \uh{j}) - f
      (t^n, x_{j+1/2}, k) \right).
  \end{align*}
  It follows then easily that
  \begin{align*}
    S \underset{\dx \to 0^+}{\longrightarrow} \
    & - \int_0^{T}\int_a^b
      \sgn^+ (u (t,x) -k)\, \left( f (t, x, u - f (t, x, k) \right)
      \, \partial_x \phi (t,x) \d{x}\d{t}
    \\
    & - \alpha \left( \int_0^{T} (u_a (t) -k)^+ \, \phi (t,a) \d{t} +
      \int_0^{T} (u_b (t) -k)^+ \, \phi (t,b) \d{t}\right).
  \end{align*}
  Let us rewrite $S$ in~\eqref{eq:S} as follows:
  \begin{align*}
    S = \
    & - \dt \sum_{n=0}^{N_T-1} \sum_{j=1}^N G^{n,k}_{j+1/2}
      (\uh{j}, \uh{j}) \left(\phi (t^n,x_{j+1}) - \phi (t^n,x_j\right)
    \\
    &- \alpha \, \dt \sum_{n=0}^{N_T-1} \left( (\uh{a} -k)^+ \, \phi
      (t^n,a) + (\uh{b}-k)^+ \, \phi (t^n,b) \right)
    \\
    = \
    & \dt \sum_{n=0}^{N_T-1} \left( \sum_{j=1}^N
      G^{n,k}_{j+1/2}(\uh{j}, \uh{j}) \, \phi (t^n, x_{j+1}) -
      \sum_{j=0}^{N-1} G^{n,k}_{j+3/2} (\uh{j+1}, \uh{j+1}) \, \phi
      (t^n, x_{j+1}) \right)
    \\
    &- \alpha \, \dt \sum_{n=0}^{N_T-1} \left( (\uh{a} -k)^+ \, \phi
      (t^n,a) + (\uh{b}-k)^+ \, \phi (t^n,b) \right)
    \\
    = \
    & -\dt \sum_{n=0}^{N_T-1} \sum_{j=1}^{N-1}\left(
      G^{n,k}_{j+1/2}(\uh{j}, \uh{j}) - G^{n,k}_{j+3/2} (\uh{j+1},
      \uh{j+1}) \right) \, \phi (t^n, x_{j+1})
    \\
    & - \dt \sum_{n=0}^{N_T-1} \left( G^{n,k}_{N+1/2}(\uh{N}, \uh{N})
      \, \phi (t^n, x_{N+1}) - G^{n,k}_{3/2} (\uh{1}, \uh{1}) ) \,
      \phi (t^n, x_{1}) \right)
    \\
    &- \alpha \, \dt \sum_{n=0}^{N_T-1} \left( (\uh{a} -k)^+ \, \phi
      (t^n,a) + (\uh{b}-k)^+ \, \phi (t^n,b) \right)
    \\
    =\
    & S^{int} + S^b,
  \end{align*}
  where we set
  \begin{align*}
    S^{int} = \
    & -\dt \sum_{n=0}^{N_T-1} \sum_{j=1}^{N-1}\left(
      G^{n,k}_{j+1/2}(\uh{j}, \uh{j}) - G^{n,k}_{j+3/2} (\uh{j+1},
      \uh{j+1}) \right) \, \phi (t^n, x_{j+1}),
    \\
    S^b = \
    & - \dt \sum_{n=0}^{N_T-1} \left( G^{n,k}_{N+1/2}(\uh{N},
      \uh{N}) \, \phi (t^n, x_{N+1}) - G^{n,k}_{3/2} (\uh{1}, \uh{1})
      ) \, \phi (t^n, x_{1}) \right)
    \\
    &- \alpha \, \dt \sum_{n=0}^{N_T-1} \left( (\uh{a} -k)^+ \, \phi
      (t^n,a) + (\uh{b}-k)^+ \, \phi (t^n,b) \right).
  \end{align*}
  Focus on $S^{int}$: by adding and subtracting
  $G^{n,k}_{j+1/2} (\uh{j}, \uh{j})$ in the first brackets, we can
  rewrite this term as
  \begin{align*}
    S^{int} = \
    & -\dt \sum_{n=0}^{N_T-1} \sum_{j=1}^{N-1}\left(
      G^{n,k}_{j+1/2}(\uh{j}, \uh{j+1}) - G^{n,k}_{j+1/2} (\uh{j},
      \uh{j}) \right) \, \phi (t^n, x_{j+1})
    \\
    & -\dt \sum_{n=0}^{N_T-1} \sum_{j=1}^{N-1}\left(
      G^{n,k}_{j+1/2}(\uh{j}, \uh{j+1}) - G^{n,k}_{j+3/2} (\uh{j+1},
      \uh{j+1}) \right) \, \phi (t^n, x_{j+1}).
  \end{align*}
  We evaluate now the distance between $T^{int}$ and $S^{int}$:
  \begin{displaymath}
    \modulo{T^{int} - S^{int}} \leq
    \dt \sum_{n=0}^{N_T-1} \sum_{j=1}^{N-1}
    \modulo{ G^{n,k}_{j+1/2}(\uh{j}, \uh{j+1})
      - G^{n,k}_{j+1/2}  (\uh{j}, \uh{j})}\,
    \modulo{\phi (t^n, x_{j+1}) - \phi (t^n, x_j)}.
  \end{displaymath}
  Observe that
  \begin{align*}
    & \modulo{ G^{n,k}_{j+1/2}(\uh{j}, \uh{j+1}) - G^{n,k}_{j+1/2}
      (\uh{j}, \uh{j})}
    \\
    = \
    & \modulo{ F^n_{j+1/2}(\uh{j} \wedge k, \uh{j+1} \wedge k) -
      F^n_{j+1/2} (\uh{j} \wedge k, \uh{j}) \wedge k}
    \\
    =\
    & \frac12 \left|f (t^n, x_{j+1/2},\uh{j}\wedge k) + f (t^n,
      x_{j+1/2}, \uh{j+1} \wedge k)\right.
    \\
    & \qquad \left.  - 2 \, f (t^n, x_{j+1/2}, \uh{j} \wedge k) -
      \alpha \, ( \uh{j+1} \wedge k - \uh{j} \wedge k)\right|
    \\
    = \
    & \frac12 \left|f (t^n, x_{j+1/2},\uh{j+1}\wedge k) - f (t^n,
      x_{j+1/2}, \uh{j} \wedge k)\right.
    \\
    & \qquad\qquad \left. - \alpha \, ( \uh{j+1} \wedge k - \uh{j}
      \wedge k)\right|
    \\
    \leq \
    & \frac12\, (L_f (t^n) + \alpha) \, \modulo{ \uh{j+1}
      \wedge k - \uh{j} \wedge k}
    \\
    \leq \
    & \alpha \, \modulo{ \uh{j+1} - \uh{j}}.
  \end{align*}
  Therefore,
  \begin{align}
    \nonumber \modulo{T^{int} - S^{int}} \leq \
    & \alpha \, \dx \, \dt
      \, \norma{\partial_x \phi}_{\L\infty} \sum_{n=0}^{N_T-1}
      \sum_{j=1}^{N-1} \modulo{ \uh{j+1} - \uh{j}}
    \\
    \label{eq:int}
    \leq \
    & \alpha \, \dx \, T \, \norma{\partial_x \phi}_{\L\infty}
      \max_{0 \leq n \leq N_T-1} \tv (u_\Delta (t^n)) = \mathcal{O} (\dx),
  \end{align}
  thanks to the uniform BV estimate~\eqref{eq:spaceBV}. Pass now to
  the terms $T^b$ and $S^b$:
  \begin{align}
    \nonumber S^b - T^b = \
    & - \dt \sum_{n=0}^{N_T-1} \left(
      G^{n,k}_{N+1/2}(\uh{N}, \uh{N}) \, \phi (t^n, x_{N+1}) -
      G^{n,k}_{3/2} (\uh{1}, \uh{1}) ) \, \phi (t^n, x_{1}) \right)
    \\
    \nonumber
    &- \alpha \, \dt \sum_{n=0}^{N_T-1} \left( (\uh{a} -k)^+
      \, \phi (t^n,a) + (\uh{b}-k)^+ \, \phi (t^n,b) \right)
    \\
    \nonumber
    &- \dt \sum_{n=0}^{N_T-1} \left(G^{n,k}_{N+1/2} (\uh{N},
      \uh{b}) - G^{n,k}_{N+1/2} (\uh{N}, \uh{N})\right) \, \phi (t^n,
      x_N)
    \\ \nonumber
    & + \dt \sum_{n=0}^{N_T-1} \left(G^{n,k}_{1/2}
      (\uh{a}, \uh{1}) - G^{n,k}_{3/2} (\uh{1}, \uh{1}) \right) \,
      \phi (t^n, x_{1})
    \\ \label{eq:10} = \
    & \dt \sum_{n=0}^{N_T-1}
      \left( G^{n,k}_{1/2} (\uh{a}, \uh{1}) \, \phi (t^n, x_{1}) -
      \alpha \, (\uh{a} -k)^+ \, \phi (t^n,a) \right)
    \\ \label{eq:10b}
    & - \dt \sum_{n=0}^{N_T-1} \left( \alpha \,
      (\uh{b}-k)^+ \, \phi (t^n,b) + G^{n,k}_{N+1/2} (\uh{N}, \uh{b})
      \, \phi (t^n, x_N)\right)
    \\
    \label{eq:10c}
    & - \dt \sum_{n=0}^{N_T-1}G^{n,k}_{N+1/2}(\uh{N}, \uh{N})
      \left(\phi (t^n, x_{N+1})- \phi (t^n, x_N)\right).
  \end{align}
  Observe that
  \begin{align*}
    \frac{\partial F^n_{j+1/2}}{\partial u} (u,v) = \
    & \frac12 \,
      \left(
      \partial_u f (t^n, x_{j+1/2}, u) + \alpha \right) \geq \frac12
      \, (-L_f (t^n) + \alpha) \geq 0,
    \\
    \frac{\partial F^n_{j+1/2}}{\partial v} (u,v) = \
    & \frac12 \,
      \left(
      \partial_u f (t^n, x_{j+1/2}, v) - \alpha \right) \leq \frac12
      \, (L_f (t^n) -\alpha) \leq 0,
  \end{align*}
  meaning that the numerical flux is increasing with respect to the
  first variable and decreasing with respect to the second one. Thus,
  \begin{align*}
    G^{n,k}_{j+1/2} (u,v) = \
    & F^n_{j+1/2} (u \wedge k, v \wedge k) -
      F^n_{j+1/2} (k,k)
    \\
    \geq \
    & F^n_{j+1/2} (k, v \wedge k) - F^n_{j+1/2} (k,k)
    \\
    = \
    & \frac12 \, \left(f (t^n, x_{j+1/2}, v \wedge k) - f (t^n,
      x_{j+1/2},k) -\alpha (v \wedge k - k) \right)
    \\
    \geq \
    & - \frac {L_f (t^n) +\alpha}{2} \, \modulo{v \wedge k - k}
    \\
    \geq \
    & - \alpha \, (v-k)^+
  \end{align*}
  and
  \begin{align*}
    G^{n,k}_{j+1/2} (u,v) = \
    & F^n_{j+1/2} (u \wedge k, v \wedge k) -
      F^n_{j+1/2} (k,k)
    \\
    \leq \
    & F^n_{j+1/2} (u \wedge k, k) - F^n_{j+1/2} (k,k)
    \\
    = \
    & \frac12 \, \left(f (t^n, x_{j+1/2}, u \wedge k)
      - f (t^n, x_{j+1/2},k) -\alpha (k - u \wedge k ) \right)
    \\
    \leq \
    & \frac {L_f (t^n) +\alpha}{2} \, \modulo{u \wedge k - k}
    \\
    \leq \
    & \alpha \, (u-k)^+.
  \end{align*}
  Hence,
  \begin{align*}
    [\eqref{eq:10}] = \
    & \dt \sum_{n=0}^{N_T-1} G^{n,k}_{1/2}
      (\uh{a}, \uh{1}) \left(\phi (t^n, x_{1}) - \phi (t^n,a)\right)
    \\
    &+ \dt \sum_{n=0}^{N_T-1} \left( G^{n,k}_{1/2} (\uh{a}, \uh{1}) -
      \alpha \, (\uh{a} -k)^+ \right) \phi (t^n,a)
    \\
    \leq \
    & \alpha \, T \, \dx \, \norma{\partial_x
      \phi}_{\L\infty}\!  \sup_{ n \in [0,N_T-1]}(\uh{a} - k)^+\! +
      \alpha \, \dt \sum_{n=0}^{N_T-1} \left( (\uh{a} - k)^+ \!- (\uh{a}
      - k)^+\!\right) \phi (t^n,a)
    \\
    \leq \
    & \alpha \, T \, \dx \, \norma{\partial_x \phi}_{\L\infty}
      \norma{u_a}_{\L\infty ([0,T])} = \mathcal{O} (\dx),
    \\
    [\eqref{eq:10b}] = \
    & - \dt \sum_{n=0}^{N_T-1} \left( \alpha \,
      (\uh{b} -k)^+ + G^{n,k}_{N+1/2} (\uh{N}, \uh{b})\right) \phi
      (t^n,b)
    \\
    & - \dt \sum_{n=0}^{N_T-1} G^{n,k}_{N+1/2} (\uh{N}, \uh{b})
      \left(\phi (t^n, x_{N}) - \phi (t^n,b)\right)
    \\
    \leq \
    & \!- \alpha \, \dt\! \sum_{n=0}^{N_T-1}\! \left( (\uh{b} -
      k)^+ \!- (\uh{b} - k)^+\!\right) \phi (t^n,b) + \alpha \, T \,
      \dx \, \norma{\partial_x \phi}_{\L\infty}\!\!\! \sup_{n \in
      [0,N_T-1]}(\uh{b} - k)^+
    \\
    \leq \
    & \alpha \, T \, \dx \, \norma{\partial_x \phi}_{\L\infty}
      \norma{u_b}_{\L\infty ([0,T])} = \mathcal{O} (\dx),
    \\
    [\eqref{eq:10c}] = \
    & \dt \modulo{
      \sum_{n=0}^{N_T-1}G^{n,k}_{N+1/2}(\uh{N}, \uh{N}) \left(\phi
      (t^n, x_{N+1})- \phi (t^n, x_N)\right) }
    \\
    \leq \
    & \dt \, \dx \, \norma{\partial_x \phi}_{\L\infty}
      \sum_{n=0}^{N_T-1} \modulo{G^{n,k}_{N+1/2}(\uh{N}, \uh{N}) }
    \\
    = \
    &\dt \, \dx \, \norma{\partial_x \phi}_{\L\infty}
      \sum_{n=0}^{N_T-1} \modulo{F^n_{N+1/2} (\uh{N} \wedge k, \uh{N}
      \wedge k) - F^N_{N+1/2} (k,k)}
    \\
    = \
    & \dt \, \dx \, \norma{\partial_x \phi}_{\L\infty}
      \sum_{n=0}^{N_T-1} \modulo{f (t^n, x_{N+1/2}, \uh{N} \wedge k) - f
      (t^n, x_{N+1/2}, k)}
    \\
    \leq  \
    & L_f (T) \, \dt \, \dx \, \norma{\partial_x \phi}_{\L\infty}
      \sum_{n=0}^{N_T-1} \modulo{\uh{N} \wedge k - k}
    \\
    = \
    & L_f (T) \, \dt \, \dx \, \norma{\partial_x \phi}_{\L\infty}
      \sum_{n=0}^{N_T-1} (\uh{N} - k)^+
    \\
    \leq \
    & L_f (T) \, T \, \dx \, \norma{\partial_x \phi}_{\L\infty}
      \sup_{0\leq n \leq N_T-1} \norma{u^n }_{\L\infty} = \mathcal{O}
      (\dx),
  \end{align*}
  thanks to the uniform $\L\infty$ estimate~\eqref{eq:linf}.  Hence,
  $S^b - T^b \leq \mathcal{O} (\dx)$, so that we finally get
  \begin{align*}
    0 \geq \
    & [\eqref{eq:9}\ldots\eqref{eq:9c}]
    \\
    = \
    & [\eqref{eq:9}] + [\eqref{eq:9c}] + T \pm S
    \\
    \geq \
    & [\eqref{eq:9}] + [\eqref{eq:9c}] + S - \mathcal{O} (\dx),
  \end{align*}
  concluding the proof.
\end{proof}

\subsection{Uniqueness}
\label{sec:uniq}

The uniqueness of the solution to the IBVP~\eqref{eq:original} follows
from the Lipschitz continuous dependence of the solution on initial
and boundary data, proved for the multidimensional case
in~\cite[Theorem~4.3]{ColomboRossi2015}.

\begin{proposition}{\bf{(Lipschitz continuous dependence on initial
      and boundary data)}}
  \label{prop:lipdata}
  Let~\ref{f} and~\ref{g} hold. Let $(u_o, \, u_a, \, u_b)$ and
  $(v_o, \, v_a, \, v_b)$ satisfy~\ref{ID}. Call $u$ and $v$ the
  corresponding solutions to the IBVP~\eqref{eq:original}. Then, for
  all $t>0$, the following estimate holds
  \begin{displaymath}
    \norma{u (t) \!- v(t)}_{\L1 (]a,b[)}\!
    \leq
    \!e^{L_g (t) \, t}
    \!\left[ \norma{u_o \!- v_o}_{\L1 (]a,b[)}\!
      + L_f (t) \!\left(\norma{u_a \!- v_a}_{\L1 ([0,t])}\!
        + \norma{u_b\! - v_b}_{\L1 ([0,t])} \right)\!\right],
  \end{displaymath}
  where $L_f (t)$ and $L_g (t)$ are defined in~\eqref{eq:lipconst}.
\end{proposition}

\section{Proofs of Theorem~\ref{thm:main} and Theorem~\ref{thm:stab}}
\label{sec:stabflux}

\begin{proofof}{Theorem~\ref{thm:main}}
  The existence of a unique solution to the IBVP~\eqref{eq:original}
  is ensured by the results presented in Section~\ref{sec:existence},
  see in particular \S~\ref{sec:conv} and \S~\ref{sec:uniq}.

  The estimates on the solution to the IBVP~\eqref{eq:original} are
  obtained by passing to the limit in the corresponding discrete
  estimates, namely~\eqref{eq:linf} for the $\L\infty$--bound,
  \eqref{eq:spaceBV} for the bound on the total variation,
  and~\eqref{eq:32}--\eqref{eq:34} for the Lipschitz continuity in
  time.
\end{proofof}

\begin{proofof}{Theorem~\ref{thm:stab}}
  By Lemma~\ref{lem:linf} we have, for all $t>0$
  \begin{align}
    \label{eq:20}
    \norma{u_1 (t)}_{\L\infty (]a,b[)} \leq \
    & \mathcal{U}_1 (t),
    &
      \norma{u_2 (t)}_{\L\infty (]a,b[)} \leq \
    & \mathcal{U}_2 (t),
  \end{align}
  the definition of $\mathcal{U}_i (t)$, $i=1,2$, following
  from~\eqref{eq:Ut}. Introduce, moreover, the following notation: for
  all $t>0$
  \begin{align}
    \label{eq:30}
    U_i (t) = \
    & [ -\mathcal{U}_i (t) , \, \mathcal{U}_i (t)], \quad
      i=1,2
    & U (t) = \
    & U_1 (t) \cup U_2 (t).
  \end{align}

  We apply the \emph{doubling of variables method} introduced by
  Kru\v{z}kov~\cite{Kruzkov}, following also~\cite[proof of
  Theorem~4.3]{ColomboRossi2015}. Let
  $\phi \in \Cc1 (]0,T[ \times \reali; \reali^+)$ be a test function
  as in Definition~\ref{def:solBLN} with
  \begin{equation}
    \label{eq:phi}
    \phi (t,x) = 0 \mbox{ for all } t \in \, ]0,T[ \mbox{ and }
    x \in [a,a+h_*] \cup [b-h_*, b]
  \end{equation}
  for a positive $h_*$. Clearly, $\phi (0,x)=0$ for all $x\in\reali$.

  Let $Y \in \Cc\infty (\reali;\reali^+)$ be such that
  \begin{align*}
    Y (-z) = \
    & Y (z),
    & Y(z) = \
    & 0 \mbox{ for } |z|\geq 1,
    &
      \int_\reali Y (z) \d{z} = \
    & 1,
  \end{align*}
  and define, for $h \in \,]0,h_*[$,
  $Y_h (z) = \frac{1}{h} Y \left(\frac{z}{h}\right)$. Clearly,
  $Y_h \in \Cc\infty (\reali;\reali^+)$, $Y_h (-z) = Y_h (z)$,
  $Y_h (z) = 0$ for $|z| \geq h$ and $\int_\reali Y_h (z) \d{z} =
  1$. Moreover, $Y_h \to \delta_0$ as $h \to 0$, $\delta_0$ being the
  Dirac delta in $0$. Define
  \begin{displaymath}
    \psi_h (t,x,s,y) =
    \phi\left(\frac{t+s}2, \frac{x+y}2\right) \, Y_h (t-s) \, Y_h (x-y).
  \end{displaymath}
  For the sake of simplicity, introduce the space
  $\Pi_T = \, ]0,T[\, \times\, ]a,b[$. Since $u_1= u_1 (t,x)$ and
  $u_2= u_2 (s,y)$ are solutions to the IBVPs~\eqref{eq:original}, we
  derive the following entropy inequalities
  \begin{align*}
    &\iq \left\{ \modulo{u_1 (t,x) - u_2 (s,y)} \partial_t \psi_h
      (t,x,s,y)\right.
    \\[-8pt]
    &\qquad\qquad + \sgn\left(u_1 (t,x) - u_2 (s,y)\right) \left[
      f_1\left(t, x, u_1 (t,x)\right) - f_1 \left(t, x, u_2
      (s,y)\right) \right] \, \partial_x \psi_h (t,x,s,y)
    \\
    & \qquad\qquad \left.  + \sgn\left(u_1 (t,x) -u_2 (s,y)\right)
      \left( g_1 \left(t,x,u_1 (t,x)\right) - \partial_x f_1 \left(t,
      x, u_2 (s,y)\right) \right)\, \psi_h (t,x,s,y) \right\}
    \\
    & \hspace{13cm} \times \d{x} \d{t} \d{y} \d{s}
    \\
    & +\id\!\! \int_{a}^{b} \modulo{u_o(x)-u_2 (s,y)} \psi_h (0,x,s,y)
      \d{x} \d{y} \d{s} \geq 0
  \end{align*}
  and
  \begin{align*}
    &\iq \left\{ \modulo{u_2 (s,y) - u_1 (t,x)} \partial_s \psi_h
      (t,x,s,y)\right.
    \\[-8pt]
    &\qquad\qquad + \sgn\left(u_2 (s,y) - u_1 (t,x)\right) \left[
      f_2\left(s, y, u_2 (s,y)\right) - f_2\left(s, y, u_1
      (t,x)\right) \right] \, \partial_y \psi_h (t,x,s,y)
    \\
    & \qquad\qquad \left.  + \sgn\left(u_2 (s,y) - u_1 (t,x)\right)
      \left( g_2 \left(s,y,u_2 (s,y)\right) - \partial_y f_2 \left(s,
      y, u_1 (t,x)\right)\right) \, \psi_h (t,x,s,y) \right\}
    \\
    & \hspace{13cm} \times \d{x} \d{t} \d{y} \d{s}
    \\
    & +\id\!\! \int_{a}^{b} \modulo{u_o(y)-u_1 (t,x)} \psi_h (t,x,0,y)
      \d{y} \d{x} \d{t} \geq 0.
  \end{align*}
  Now combine the two inequalities above and rearrange the terms
  therein: setting $\psi_h = \psi_h (t,x,s,y)$, we obtain
  \begin{align}
    \nonumber \iq
    & \left\{ \modulo{u_1 (t,x) - u_2 (s,y)}
      \left(\partial_t \psi_h + \partial_s \psi_h \right)\right.
    \\
    \nonumber
    & + \sgn\left(u_1 (t,x) - u_2 (s,y)\right) \left[
      f_1\left(t, x, u_1 (t,x)\right) - f_1\left(s, y, u_2
      (s,y)\right) \right] \, \left( \partial_x \psi_h + \partial_y
      \psi_h \right)
    \\ \nonumber
    &+ \sgn\left(u_1 (t,x) - u_2
      (s,y)\right)
    \\ \nonumber
    &\times \left[ \left(f_1\left(s,y, u_2
      (s,y)\right) - f_1\left(t, x, u_2 (s,y)\right) \right)
      \, \partial_x \psi_h - \partial_x f_1\left(t, x, u_2
      (s,y)\right) \psi_h\right.
    \\ \nonumber
    &\qquad \left.  +
      \left( f_2 \left(s,y, u_1 (t,x)\right) - f_2 \left(t, x, u_1
      (t,x)\right) \right) \, \partial_y \psi_h + \partial_y f_2
      \left(s, y, u_1 (t,x)\right) \psi_h \right]
    \\ \nonumber
    &+
      \sgn\left(u_1 (t,x) - u_2 (s,y)\right) \left(f_1 \left(s,y, u_2
      (s,y)\right) - f_1\left(t, x, u_1 (t,x)\right)\right.
    \\
    \nonumber
    &\qquad\left. + f_2\left(t,x,u_1 (t,x)\right) - f_2
      \left(s,y,u_2 (s,y)\right) \right) \, \partial_y \psi_h
    \\
    \nonumber
    &+ \sgn\left(u_1 (t,x) - u_2 (s,y)\right) \left( g_1
      \left(t, x, u_1 (t,x)\right) - g_1 (s,y,u_2 (s,y)) \right) \,
      \psi_h
    \\ \nonumber
    & \left.  + \sgn\left(u_1 (t,x) - u_2
      (s,y)\right) \left( g_1 (s,y,u_2 (s,y)) - g_2 \left(s,y,u_2
      (s,y)\right) \right) \, \psi_h \right\} \d{x} \d{t} \d{y}
      \d{s}
    \\ \nonumber + \id\!\! \int_{a}^{b}
    & \modulo{u_o(x)-u_2
      (s,y)} \psi_h (0,x,s,y) \d{x} \d{y} \d{s}
    \\ \nonumber + \id\!\!
    \int_{a}^{b}
    & \modulo{u_o(y)-u_1 (t,x)} \psi_h (t,x,0,y) \d{y}
      \d{x} \d{t}
    \\
    \label{eq:21} = \iq
    & \sum_{j=1}^6 I_j \d{x} \d{t}
      \d{y} \d{s} + \id\!\! \int_{a}^{b} J_1 \d{x} \d{y} \d{s} + \id\!\!
      \int_{a}^{b} J_2 \d{y} \d{x} \d{t} \geq 0,
  \end{align}
  where
  \begin{align}
    \label{eq:I1}
    I_1 = \
    & \modulo{u_1 (t,x) - u_2 (s,y)} \left(\partial_t \psi_h
      + \partial_s \psi_h\right),
    \\
    \label{eq:I2}
    I_2 = \
    &
      \sgn\left(u_1 (t,x) - u_2 (s,y)\right) \left[ f_1\left(t, x, u_1
      (t,x)\right) - f_1\left(s, y, u_2 (s,y)\right) \right] \,
      \left( \partial_x \psi_h + \partial_y \psi_h \right),
    \\ \nonumber
    I_3= \
    & \sgn\left(u_1 (t,x) - u_2 (s,y)\right)
    \\ \label{eq:I3}
    &\times \left[ \left(f_1\left(s,y, u_2 (s,y)\right) - f_1\left(t,
      x, u_2 (s,y)\right) \right) \, \partial_x \psi_h
      - \partial_x f_1\left(t, x, u_2 (s,y)\right) \psi_h\right.
    \\
    \nonumber
    &\qquad \left.  + \left( f_2 \left(s,y, u_1 (t,x)\right)
      - f_2 \left(t, x, u_1 (t,x)\right) \right) \, \partial_y
      \psi_h + \partial_y f_2 \left(s, y, u_1 (t,x)\right) \psi_h
      \right],
    \\
    \label{eq:I4} I_4=\
    &\sgn\left(u_1 (t,x) - u_2
      (s,y)\right) \left(f_1 \left(s,y, u_2 (s,y)\right) - f_1\left(t,
      x, u_1 (t,x)\right)\right.
    \\ \nonumber
    &\qquad\left. +
      f_2\left(t,x,u_1 (t,x)\right) - f_2 \left(s,y,u_2 (s,y)\right)
      \right) \, \partial_y \psi_h,
    \\ \label{eq:I5} I_5=\
    &
      \sgn\left(u_1 (t,x) - u_2 (s,y)\right) \left( g_1 \left(t, x, u_1
      (t,x)\right) - g_1 (s,y,u_2 (s,y)) \right) \, \psi_h,
    \\ \label{eq:I6} I_6 = \
    & \sgn\left(u_1 (t,x) - u_2 (s,y)\right)
      \left( g_1 (s,y,u_2 (s,y)) - g_2 \left(s,y,u_2 (s,y)\right)
      \right) \, \psi_h,
    \\ \label{eq:J1} J_1 =\
    & \modulo{u_o(x)-u_2 (s,y)} \psi_h (0,x,s,y),
    \\ \label{eq:J2} J_2 = \
    &
      \modulo{u_o(y)-u_1 (t,x)} \psi_h (t,x,0,y).
  \end{align}
  Now we let $h$ go to $0$ in~\eqref{eq:21}. The integrals of
  $J_1$~\eqref{eq:J1} and $J_2$~\eqref{eq:J2}, are treated exactly as
  in~\cite[proof of Theorem~4.3]{ColomboRossi2015}:
  \begin{equation}
    \label{eq:22}
    \lim_{h \to 0}\left( \id\!\! \int_{a}^{b} J_1 \d{x} \d{y} \d{s} +
      \id\!\! \int_{a}^{b} J_2 \d{y} \d{x} \d{t}\right) =0.
  \end{equation}
  The integral of $I_1+I_2+I_5$, see~\eqref{eq:I1}, \eqref{eq:I2}
  and~\eqref{eq:I5}, is treated exactly as
  in~\cite[Theorem~1]{Kruzkov}, leading to
  \begin{align}
    \nonumber
    & \lim_{h \to 0}\iq \left\{I_1+ I_2 +I_5\right\} \d{x}
      \d{t} \d{y} \d{s}
    \\ \label{eq:23}
    =\
    & \id \left\{ \modulo{u_1
      (t,x) - u_2 (t,x)} \partial_t \phi (t,x) \right.
    \\ \nonumber
    & + \sgn\left(u_1 (t,x) - u_2 (t,x)\right) \left(f_1\left(t,x,u_1
      (t,x)\right) - f_1 \left(t,x,u_2
      (t,x)\right)\right) \partial_x \phi (t,x)
    \\ \nonumber
    &
      \left.  + \sgn\left(u_1 (t,x) - u_2 (t,x)\right) \left(
      g_1\left(t,x,u_1 (t,x)\right) - g_1 (t,x,u_2 (t,x))
      \right)\phi (t,x) \right\} \d{x} \d{t}.
  \end{align}
  The integral of $I_3$~\eqref{eq:I3} can be as well treated as
  in~\cite[Theorem~1]{Kruzkov}, bearing in mind that now we have two
  different fluxes, namely $f_1$ and $f_2$. Nevertheless, this does
  not constitute an issue in the proof. Thus we have
  \begin{equation}
    \label{eq:24}
    \lim_{h \to 0} \iq I_3 \d{x}\d{t} \d{y} \d{s} = 0.
  \end{equation}
  Consider now the integral of $I_4$~\eqref{eq:I4}:
  \begin{align}
    \nonumber
    & \iq I_4 \d{x}\d{t} \d{y} \d{s}
    \\ \nonumber
    = \
    & \iq
      \sgn\left(u_1 (t,x) - u_2 (s,y)\right)
    \\ \nonumber
    & \qquad
      \times \left[(f_1 -f_2) \left(s,y, u_2 (s,y)\right) - (f_1 - f_2)
      \left(t, x, u_1 (t,x)\right)\right]
      \partial_y \psi_h (t,x,s,y) \d{x} \d{t} \d{y} \d{s}
    \\ \label{eq:25}
    = \
    & - \iq \d{}_y \left\{ \sgn\left(u_1 (t,x) -
      u_2 (s,y)\right) \right.
    \\ \nonumber
    & \qquad \left.
      \times\left[(f_1 -f_2) \left(s,y, u_2 (s,y)\right) - (f_1 - f_2)
      \left(t, x, u_1 (t,x)\right)\right] \right\} \psi_h(t,x,s,y)
      \d{x} \d{t} \d{y} \d{s}.
  \end{align}
  Recall that for a Lipschitz function $h$ and an $\L\infty \cap \BV$
  function $z$ it holds $\d{}_y h (y,z (y)) = \partial_y h (y,z (y)) +
  \partial_z h (y,z (y)) \, z' (y) $, where we apply the chain
  rule. There, $z' (y)$ is a finite measure,
  see~\cite[Lemma~A2.1]{BouchutPerthame} and
  also~\cite[Lemma~4.1]{KarlsenRisebro}. Thus, the integral
  in~\eqref{eq:25} splits in two parts. Let us analyse them in detail.
  The first part involves the partial derivative with respect to $y$:
  \begin{align}
    \nonumber
    & \lim_{h\to 0} \iq \sgn\left(u_1 (t,x) - u_2 (s,y)\right)
      \partial_y (f_2-f_1) \!  \left(s,y,u_2 (s,y)\right)
      \, \psi_h (t,x,s,y) \d{x} \d{t} \d{y} \d{s}
    \\ \nonumber \leq \
    & \lim_{h \to 0} \iq \modulo{\partial_y (f_2-f_1)\!
      \left(s,y,u_2 (s,y)\right)}\, \psi_h (t,x,s,y) \d{x} \d{t} \d{y} \d{s}
    \\
    \label{eq:26}
    = \
    & \id \modulo{\partial_y (f_2-f_1) \! \left(s,y,u_2 (s,y)\right)} \,
      \phi (s,y) \d{y}\d{s}.
  \end{align}
  Concerning the second part, the function
  \begin{displaymath}
    H (t,x,s,y,u_1, u_2)
    = \sgn (u_1-u_2) \,
    \left[\left(f_2-f_1\right) (s,y,u_2) - \left(f_2-f_1\right) (t,x,u_1)\right]
  \end{displaymath}
  is clearly Lipschitz with respect to $u_2$, with Lipschitz constant
  $\norma{\partial_u (f_2-f_1) (s,y, \cdot)}_{\L\infty (U_2 (s))}$,
  with the notation introduced in~\eqref{eq:30}. We can thus
  apply~\cite[Lemma~A2.1]{BouchutPerthame}, see
  also~\cite[Lemma~4.1]{KarlsenRisebro}, to get
  \begin{displaymath}
    \modulo{\partial_{u_2} H \left(t,x,s,y,u_1 (t,x), u_2 (s,y)\right) \,
      \partial_y u_2 (s,y)} \leq
    \norma{\partial_u (f_2-f_1) (s,y, \cdot)}_{\L\infty (U_2 (s))} \,
    \modulo{\partial_y u_2 (s,y)}.
  \end{displaymath}
  Therefore we obtain
  \begin{align}
    \nonumber \lim_{h \to 0}
    & \iq \modulo{\partial_{u_2}H \left(t,x,s,y,u_1 (t,x), u_2 (s,y)\right)
      \, \partial_y u_2 (s,y)} \, \psi_h (t,x,s,y) \d{x} \d{t} \d{y} \d{s}
    \\
    \label{eq:27}
    \leq \
    & \id \norma{\partial_u (f_2-f_1) (t, x)}_{\L\infty (U_2 (t))} \,
      \modulo{\partial_x u_2 (t,x)} \, \phi (t,x) \d{t}\d{x}.
  \end{align}
  As far as the integral of $I_6$~\eqref{eq:I6} is concerned, we
  obtain
  \begin{equation}
    \label{eq:31}
    \iq I_6 \d{x} \d{t} \d{y} \d{s}
    \leq
    \id \modulo{(g_1 - g_2)\!\left(t,x,u_2 (t,x)\right)} \phi (t,x) \d{x} \d{t}.
  \end{equation}

  Therefore, in the limit $h \to 0$, collecting
  together~\eqref{eq:22}, \eqref{eq:23},~\eqref{eq:24}, \eqref{eq:26},
  \eqref{eq:27} and~\eqref{eq:31}, we obtain that~\eqref{eq:21}
  becomes
  \begin{align}
    \nonumber \id
    & \left\{ \modulo{u_1 (t,x) - u_2 (t,x)} \partial_t \phi (t,x) \right.
    \\[-9pt] \nonumber
    & + \sgn\left(u_1 (t,x) -
      u_2 (t,x)\right) \left(f_1\left(t,x,u_1 (t,x)\right) - f_1
      \left(t,x,u_2 (t,x)\right)\right) \partial_x \phi (t,x)
    \\
    \nonumber
    & + \sgn\left(u_1 (t,x) - u_2 (t,x)\right) \left(
      g_1\left(t,x,u_1 (t,x)\right) - g_1 (t,x,u_2 (t,x)) \right)\phi(t,x)
    \\
    \label{eq:prima}
    & + \modulo{\partial_x (f_2-f_1)\!  \left(t,x, u_2 (t,x)\right)}
      \, \phi (t,x)
    \\ \nonumber
    & + \norma{\partial_u (f_2-f_1) (t,
      x)}_{\L\infty (U_2 (t))} \, \modulo{\partial_x u_2 (t,x)} \,\phi (t,x)
    \\
    \nonumber
    & \left.  + \modulo{(g_1 - g_2)\!\left(t,x,u_2
      (t,x)\right)} \, \phi (t,x) \right\} \d{x} \d{t} \geq 0.
  \end{align}

  Following the proof of~\cite[Theorem~4.3]{ColomboRossi2015},
  introduce a function $\Phi_\epsilon \in \C2 (\reali; [0,1])$ such
  that $\Phi_\epsilon (a) = \Phi_\epsilon (b) = 1$,
  $\Phi_\epsilon (x) = 0$ for $x \in [a+\epsilon, b-\epsilon]$ and
  $\norma{\Phi'_\epsilon}_{\L\infty} \leq 1/\epsilon$. Let
  $\Psi \in \Cc2 (]0,T[; \reali^+)$, with $\Psi (0)=0$. For any
  $\epsilon>0$ sufficiently small, the function
  \begin{displaymath}
    \phi _\epsilon (t,x) = \Psi (t) \, \left(1 - \Phi_\epsilon (x)\right)
  \end{displaymath}
  satisfies~\eqref{eq:phi}, being thus an admissible test
  function. Use it in~\eqref{eq:prima} and pass to the limit as
  $\epsilon$ goes to $0$: we get
  \begin{align}
    \nonumber \id
    & \left\{ \modulo{u_1 (t,x) - u_2 (t,x)} \Psi'(t)
      \right.
    \\[-8pt] \nonumber
    & + \sgn\left(u_1 (t,x) - u_2
      (t,x)\right) \left( g_1\left(t,x,u_1 (t,x)\right) - g_1 (t,x,u_2
      (t,x)) \right) \Psi (t)
    \\
    \label{eq:dopo}
    & + \modulo{\partial_x (f_2-f_1) \left(t,x, u_2 (t,x)\right)} \,
      \Psi (t)
    \\ \nonumber
    & \left.  + \norma{\partial_u (f_2-f_1) (t,
      x)}_{\L\infty (U_2 (t))} \, \modulo{\partial_x u_2 (t,x)} \,
      \Psi (t) + \modulo{(g_1 - g_2)\!\left(t,x,u_2 (t,x)\right)} \,
      \Psi (t) \right\} \d{x} \d{t}
    \\ \nonumber
    & +\int_0^T
      \sgn\left(u_1 (t,a^+) - u_2 (t,a^+)\right) \left(f_1\left(t,a,u_1
      (t,a^+)\right) - f_1 \left(t,a,u_2 (t,a^+)\right)\right) \Psi
      (t) \d{t}
    \\
    \nonumber
    & - \int_0^T \sgn\left(u_1 (t,b^-) - u_2 (t,b^-)\right)
      \left(f_1\left(t,b,u_1 (t,b^-)\right) - f_1 \left(t,b,u_2
      (t,b^-)\right)\right) \Psi (t) \d{t} \geq 0,
  \end{align}
  where we used~\cite[Lemma~A.6 and Lemma~A.4]{ColomboRossi2015},
  recalling that the exterior normal in $a$ has negative sign.

  Introduce $\tau,\, t$ such that $0 < \tau < t < T$. Define, for
  $\ell>0$, the map
  \begin{displaymath}
    \Psi_\ell (s)  = \alpha_\ell (s-\tau-\ell) - \alpha_h (s-t-\ell),
    \mbox{ with } \alpha_\ell (\xi) = \int_{-\infty}^z Y_\ell (\zeta) \d\zeta.
  \end{displaymath}
  This function is clearly in $\Cc2 (]0,T[;\reali^+)$ and such that
  $\Psi_\ell (0)=0$. Moreover, $\Psi_\ell \to \caratt{[\tau,t]}$ and
  $\Psi'_\ell \to \delta_\tau - \delta_t $ as $\ell$ tends to
  $0$. Substitute $\Psi_\ell$ in~\eqref{eq:dopo} and pass to the limit
  $\ell \to 0$:
  \begin{align}
    \label{eq:28}
    0 \leq \
    &\int_a^b \modulo{u_1 (\tau,x) - u_2 (\tau,x)} \d{x} -
      \int_a^b \modulo{u_1 (t,x) - u_2 (t,x)} \d{x}
    \\\label{eq:28a}
    & +
      \int_\tau^t\int_a^b \norma{\partial_u g_1 (s,x, \cdot)}_{\L\infty
      (U (s))} \modulo{u_1 (s,x) -u_2 (s,x)} \d{x}\d{s}
    \\
    \label{eq:28b}
    & +\int_\tau^t\int_a^b \modulo{\partial_x (f_2-f_1)\!  \left(s,x,
      u_2 (s,x)\right)} \d{x}\d{s}
    \\ \label{eq:28c}
    & + \int_\tau
      ^t \norma{\partial_u (f_2-f_1) (s, \cdot, \cdot)}_{\L\infty
      (]a,b[\times U_2 (s))} \tv\left(u_2(s)\right) \d{s}
    \\ \label{eq:28d}
    & + \int_\tau^t \int_a^b \modulo{(g_1 -
      g_2)\!\left(s,x,u_2 (s,x)\right)} \d{x}\d{s}
    \\ \label{eq:28e}
    &
      +\int_\tau^t \sgn\left(u_1 (s,a^+) - u_2 (s,a^+)\right)
      \left(f_1\left(s,a,u_1 (s,a^+)\right) - f_1 \left(s,a,u_2
      (s,a^+)\right)\right) \d{s}
    \\
    \label{eq:28f}
    & - \int_\tau^t \sgn\left(u_1 (s,b^-) - u_2 (s,b^-)\right)
      \left(f_1\left(s,b,u_1 (s,b^-)\right) - f_1 \left(s,b,u_2
      (s,b^-)\right)\right) \d{s}.
  \end{align}

  Now we aim to find an estimate for~\eqref{eq:28e}
  and~\eqref{eq:28f}. Focus in particular on~\eqref{eq:28f}, the
  procedure being analogous for~\eqref{eq:28e}. The only contribution
  from this term come from the negative part of its argument, that is
  \begin{equation}
    \label{eq:sigma}
    B = \left[
      \sgn\left(u_1 (s,b^-) - u_2 (s,b^-)\right)
      \left(
        f_1\left(s,b,u_1 (s,b^-)\right) - f_1 \left(s,b,u_2 (s,b^-)\right)\right)
    \right]^-.
  \end{equation}
  Fix $s \in [\tau,t]$. To ease readability, in the following we will
  denote $u_1 = u_1 (s,b^-)$, $u_2 = u_2 (s,b^-)$, $u_b = u_b (s)$ and
  $f_i(z) = f_i(s,b,z)$, $i=1,2$. We have the following
  \begin{displaymath}
    B =
    \begin{cases}
      \left(f_1 (u_1 ) - f_1 (u_2 )\right)^- & \mbox{ if } u_1 > u_2 ,
      \\[3pt]
      \left(f_1 (u_1 ) - f_1 (u_2 )\right)^+ & \mbox{ if } u_1 < u_2 .
    \end{cases}
  \end{displaymath}
  Before analysing each case, recall the BLN--boundary conditions (see
  Remark~\ref{rem:blnbc}):
  \begin{align}
    \label{eq:blnf1}
    \forall k \in \mathcal{I} \left(u_b, u_1 \right): \qquad
    & \sgn (u_1 - u_b) \left(f_1 (u_1)- f_1 (k)\right) \geq 0,
    \\
    \label{eq:blnf2}
    \forall k \in \mathcal{I} \left(u_b, u_2 \right): \qquad
    & \sgn (u_2 - u_b) \left(f_2 (u_2)- f_2(k)\right) \geq 0.
  \end{align}

  \paragraph{Case $\boldsymbol{u_1>u_2}$.} In this case we have
  \begin{align*}
    B = \
    & \left(f_1 (u_1 ) - f_1 (u_2 )\right)^- = \left( f_1 (u_1)
      - f_2 (u_2)+ f_2(u_2) -f_1 (u_2) \right)^-
    \\
    \leq \
    & \left(f_2 (u_2)- f_1 (u_1)\right)^+ +
      \norma{f_2-f_1}_{\L\infty (U_2 (s))}.
  \end{align*}
  We have three sub-cases: \subparagraph{1)
    $\boldsymbol{u_b < u_2 < u_1}$.}  The
  BLN--condition~\eqref{eq:blnf1} now reads $f_1 (u_1) \geq f_1 (k)$
  for all $k \in [u_b,u_1]$ . The choice $k=u_2$ is admissible:
  $f_1 (u_1) \geq f_1 (u_2)$. Thus
  \begin{displaymath}
    f_2 (u_2) - f_1 (u_1) \leq f_2 (u_2) - f_1 (u_2)
    \leq
    \norma{f_2-f_1}_{\L\infty (U_2 (s))}.
  \end{displaymath}

  \subparagraph{2) $\boldsymbol{u_2 < u_b < u_1}$.} The
  BLN--conditions~\eqref{eq:blnf1}--\eqref{eq:blnf2} now reads
  \begin{align*}
    \forall k \in [u_b, u_1]: \qquad & f_1 (u_1) \geq f_1 (k),
    \\
    \label{eq:blnf2}
    \forall k \in [u_2, u_b]: \qquad & f_2 (u_2) \leq f_2(k).
  \end{align*}
  In both cases, the choice $k=u_b$ is admissible, yielding
  $f_1 (u_1) \geq f_1 (u_b)$ and $f_2 (u_2) \leq f_2 (u_b)$. Thus
  \begin{displaymath}
    f_2 (u_2) - f_1 (u_1) \leq f_2 (u_b) - f_1 (u_b).
  \end{displaymath}

  \subparagraph{3) $\boldsymbol{u_2 < u_1 < u_b}$.}The
  BLN--condition~\eqref{eq:blnf2} now reads $f_2 (u_2) \leq f_2 (k)$
  for all $k \in [u_2, u_b]$ . The choice $k=u_1$ is admissible:
  $f_2 (u_2) \leq f_2 (u_1)$. Thus
  \begin{displaymath}
    f_2 (u_2) - f_1 (u_1) \leq f_2 (u_1) - f_1 (u_1)
    \leq
    \norma{f_2-f_1}_{\L\infty (U_1 (s))}.
  \end{displaymath}

  \paragraph{Case $\boldsymbol{u_1<u_2}$.} The proof is done in a
  similar way.

  \medskip

  We can thus conclude that, for $s \in [\tau,t]$,
  $ B \leq 2 \, \norma{(f_2-f_1)(s,b, \cdot)}_{\L\infty(U (s))}$,
  where $U (s)$ is as in~\eqref{eq:30}, so that, going back
  to~\eqref{eq:28f},
  \begin{displaymath}
    -\int_\tau^t B \d{s} \leq \int_\tau^t 2 \,
    \norma{(f_2-f_1)(s,b, \cdot)}_{\L\infty(U (s))} \d{s}.
  \end{displaymath}
  Similarly, we obtain
  \begin{displaymath}
    [\eqref{eq:28e}] \leq \int_\tau^t 2 \,
    \norma{(f_2-f_1)(s,a, \cdot)}_{\L\infty(U (s))} \d{s}.
  \end{displaymath}
  Therefore, we get the following estimate
  for~\eqref{eq:28}--\eqref{eq:28f}:
  \begin{align*}
    & \int_a^b\modulo{u_1 (t,x) - u_2 (t,x)} \d{x}
    \\ \leq \
    & \int_a^b\modulo{u_1 (\tau,x) - u_2 (\tau,x)} \d{x} + \int_\tau^t
      \int_a^b \norma{\partial_u g_1 (s,x, \cdot)}_{\L\infty (U (s))}
      \modulo{u_1 (s,x)- u_2 (s,x)} \d{x}\d{s}
    \\
    & + \int_\tau^t \int_a^b \norma{\partial_x (f_2-f_1) (s,x,
      \cdot)}_{\L\infty (U_2 (s))} \d{x}\d{s} + \int_\tau^t \int_a^b
      \norma{(g_1-g_2) (s,x, \cdot)}_{\L\infty (U_2 (s))} \d{x}\d{s}
    \\
    & + \int_\tau^t \norma{\partial_u (f_2-f_1) (s, \cdot,
      \cdot)}_{\L\infty (]a,b[\times U_2 (s))} \, \tv \left(u_2
      (s)\right) \d{s}
    \\
    & + 2 \int_\tau^t \norma{(f_2-f_1)(s,a, \cdot)}_{\L\infty(U (s))}
      \d{s} +2 \int_\tau^t \norma{(f_2-f_1)(s,b, \cdot)}_{\L\infty(U
      (s))} \d{s}.
  \end{align*}
  Let now $\tau$ go to $0$, recalling that the initial datum is the
  same, and apply Gronwall lemma:
  \begin{align*}
    & \int_a^b\modulo{u_1 (t,x) - u_2 (t,x)} \d{x}
    \\
    \leq \
    & e^{t \,
      \norma{\partial_u g_1}_{\L\infty ([0,t]\times[a,b] \times U
      (t))}} \left( \int_0^t \int_a^b \norma{\partial_x (f_2-f_1)
      (s,x, \cdot)}_{\L\infty (U_2 (s))} \d{x}\d{s} \right.
    \\
    & + \int_0^t \int_a^b \norma{(g_1-g_2) (s,x, \cdot)}_{\L\infty
      (U_2 (s))} \d{x}\d{s}
    \\
    & + \int_0^t \norma{\partial_u (f_2-f_1) (s, \cdot,
      \cdot)}_{\L\infty (]a,b[\times U_2 (s))} \, \tv \left(u_2
      (s)\right) \d{s}
    \\
    & \left. + 2 \int_0^t \norma{(f_2-f_1)(s,a, \cdot)}_{\L\infty(U
      (s))} \d{s} +2 \int_0^t \norma{(f_2-f_1)(s,b,
      \cdot)}_{\L\infty(U (s))} \d{s}\right).
  \end{align*}

  Exchanging the role $u_1$ and $u_2$, and thus that of $f_1, \, g_1$
  and $f_2, \, g_2$, we get a symmetric estimate. Therefore, recalling
  the definition of $U (t)$~\eqref{eq:30}, the final estimate reads
  \begin{align*}
    & \int_a^b\modulo{u_1 (t,x) - u_2 (t,x)} \d{x}
    \\
    \leq \
    &
      \exp\left(t \, \min\left\{\norma{\partial_u g_1}_{\L\infty
      ([0,t]\times[a,b] \times U (t))}, \, \norma{\partial_u
      g_2}_{\L\infty ([0,t]\times[a,b] \times U (t))}\right\}
      \right)
    \\
    & \times \left( \int_0^t \int_a^b \norma{\partial_x (f_2-f_1)
      (s,x, \cdot)}_{\L\infty (U(s))} \d{x}\d{s} \right.
    \\
    & \qquad+ \int_0^t \int_a^b \norma{(g_1-g_2) (s,x,
      \cdot)}_{\L\infty (U (s))} \d{x}\d{s}
    \\
    & \qquad+ \int_0^t \norma{\partial_u (f_2-f_1) (s, \cdot,
      \cdot)}_{\L\infty (]a,b[\times U (s))} \, \min\left\{\tv
      \left(u_1 (s)\right), \,\tv \left(u_2 (s)\right)\right\} \d{s}
    \\
    & \left. \qquad+ 2 \int_0^t \norma{(f_2-f_1)(s,a,
      \cdot)}_{\L\infty(U (s))} \d{s} +2 \int_0^t
      \norma{(f_2-f_1)(s,b, \cdot)}_{\L\infty(U (s))} \d{s}\right).
  \end{align*}
\end{proofof}

\section*{Acknowledgements}
The author would like to thank Rinaldo M. Colombo and Paola Goatin for
fruitful discussions, and Boris Andreianov for his suggestion about
the stability estimate.

\small{ \bibliography{1d}

  \bibliographystyle{abbrv} }

\end{document}